\documentclass[11pt]{article}
\usepackage[margin=1.0in]{geometry}
\usepackage[dvips]{graphicx}
\usepackage{color,amsmath,latexsym,amssymb,amsthm}
\usepackage{mathtools}
\usepackage{xspace}
\usepackage{pgf,tikz}
\usepackage{rotating}
\usepackage{multirow}
\usepackage{algorithmic}
\usepackage{framed}
\usepackage{subfig}
\usepackage{setspace}
\usepackage{appendix}
\usepackage{ifthen}
\usepackage{mathtools}
\usepackage{mathrsfs}
\usepackage{hyperref}
\usepackage{cleveref}
\usepackage{enumitem}
\usepackage{algorithm, algorithmic}
\usepackage{soul}

\usepackage[colorinlistoftodos,textsize=scriptsize]{todonotes}
\colorlet{todored}{red}
\colorlet{todoblue}{blue}
\colorlet{todoorange}{orange}
\colorlet{todopurple}{purple}

\newtheorem{theorem}{Theorem}
\newtheorem{lemma}[theorem]{Lemma}
\newtheorem{assumption}[theorem]{Assumption}
\newtheorem{proposition}[theorem]{Proposition}
\newtheorem{remark}[theorem]{Remark}

\newtheorem{corollary}[theorem]{Corollary}
\newtheorem{definition}[theorem]{Definition}

\Crefrangeformat{enumi}{Assumptions #3#1#4--#5#2#6}
\Crefname{assumption}{Assumption}{Assumptions}
\Crefrangeformat{assumption}{Assumptions #3#1#4--#5#2#6}
\Crefformat{assumption}{Assumption~\textnormal{#2#1#3}}
\Crefformat{theorem}{Theorem~\textnormal{#2#1#3}}
\Crefformat{proposition}{Proposition~\textnormal{#2#1#3}}
\Crefformat{lemma}{Lemma~\textnormal{#2#1#3}}
\Crefformat{corollary}{Corollary~\textnormal{#2#1#3}}

\newcommand{\risk}{\mathcal{R}}

\newcommand{\cA}{\mathcal{A}}

\newcommand{\cF}{\mathcal{F}}

\newcommand{\real}{\mathbb{R}}

\newcommand{\embedding}{\hookrightarrow}
\newcommand{\funArg}[1]{\ifthenelse{\equal{#1}{}}  
	{}
	{[#1]}
}
\newcommand{\FunArg}[1]{\ifthenelse{\equal{#1}{}}  
	{}
	{({#1})}
}

\newcommand{\AVaR}[2][]{\mathrm{AVaR}_\beta\funArg{#2}}
\newcommand{\epiAVaR}[3][]{\mathrm{AVaR}_\beta^{#3}\funArg{#2}}
\newcommand{\smoothAVaR}[3][]{\sigma_\beta^{#3}\funArg{#2}}
\newcommand{\dist}[2]{\mathrm{dist}\FunArg{#1,#2}}
\newcommand{\deviation}[2]{\mathbb{D}\FunArg{#1,#2}}
\newcommand{\Risk}[2][]{\risk\funArg{#2}}
\newcommand{\dom}[2][]{\operatorname*{dom}\FunArg{#2}}
\newcommand{\bochner}[3][]{L^{#2}(\Xi, \cA, \mathbb{P}; #3)}
\newcommand{\bochnerreal}[3][]{L^{#2}(#3)}
\newcommand{\Lp}[1]{L^{#1}(\Omega, \cF, P)}

\newcommand{\tand}{\text{and}}

\newcommand{\pobj}{f}

\newcommand{\csp}{Z}
\newcommand{\fun}{h}
\newcommand{\du}{\ensuremath{\mathrm{d}}}
\newcommand{\adcsp}{Z_{\rm ad}}
\newcommand{\B}{\mathbf{B}}
\newcommand{\rhs}{\mathbf{b}}
\newcommand{\A}{\mathbf{A}}
\newcommand{\ssp}{U}

\newcommand{\wpone}{w.p.~$1$\xspace}

\newcommand{\inner}[3][]{( #2, #3 )_{#1}}

\newcommand{\spL}[2]{\mathscr{L}(#1, #2)}
\newcommand{\norm}[2][2]{\|#2\|_{#1}}

\newcommand{\Caratheodory}{Carath\'eodory}

\renewcommand{\natural}{\mathbb{N}}

\newcommand{\dualp}[3][]{\langle #2, #3 \rangle_{{#1}^*, #1}}
\newcommand{\dualpHzeroone}[3][]
{\langle #2, #3 \rangle_{H^{-1}(#1), H_0^1(#1)}}
\newcommand{\maxo}[2][]{(#2)^+}
\newcommand{\wto}{\rightharpoonup}
\newcommand{\domain}{D}

\DeclarePairedDelimiterXPP\cE[1]{\mathbb{E}}[]{}{%
	
	#1}

\renewcommand{\abstract}[1]
{
	{\small
		\textbf{Abstract.} {#1}
		\\
	}
}

\newcommand{\keywords}[1]
{
	{\small
		\textbf{Key words.} {#1}
		\\
	}
}

\newcommand{\amssubject}[1]
{
	{\small
		\textbf{AMS subject classifications.} {#1}
	}
}

\title{Asymptotic Consistency for Nonconvex Risk-Averse Stochastic Optimization with Infinite Dimensional Decision Spaces}

\author{Johannes Milz%
        \thanks{H.\ Milton Stewart School of Industrial and Systems Engineering, Georgia Institute of Technology, Atlanta, GA 30332,     USA        \texttt{johannes.milz@isye.gatech.edu}.}
     \and
        Thomas M.\ Surowiec \hspace{-2.5mm}
        \thanks{Department of Numerical
        	Analysis and Scientific Computing, Simula Research Laboratory, 0164 Oslo, Norway, \texttt{thomasms@simula.no}.}
            }
        
\date{March 31, 2023}

\begin{document}

\maketitle

\abstract{%
Optimal values and solutions of empirical approximations of stochastic optimization problems can be viewed as statistical estimators of their  true values. From this perspective, it is important to understand the asymptotic behavior of these estimators as the sample size goes to infinity. This area of study has a long tradition in stochastic programming. However, the literature is lacking consistency analysis for problems in which the decision variables are taken from an infinite dimensional space, which arise in optimal control, scientific machine learning, and statistical estimation. By exploiting the typical problem structures found in these applications that give rise to hidden norm compactness properties for solution sets,  we prove consistency results for nonconvex risk-averse stochastic optimization problems formulated in infinite dimensional space. The proof is based on  several crucial results from the theory of  variational convergence. The theoretical results are demonstrated for several important problem classes arising in the literature.
}

\par
\par
\amssubject{%
	90C15, 90C06, 62F12, 35Q93, 49M41, 49J52
}

\section{Introduction}\label{sec:intro}
The asymptotic behavior of empirical approximations is a central point of study in optimization under uncertainty.  There is a long tradition going back to the fundamental contributions \cite{Huber1967,Dupacova1988,Shapiro1989,WRoemisch_RSchultz_1991,Shapiro1991,King1993,Shapiro1993,Pflug1995,Shapiro2000c,STRachev_WRoemisch_2002,Pflug2003,Lachout2005,Shapiro2005,Robinson1996}.  These works have since given rise to standard derivation techniques for problems with finite dimensional decision spaces. There are in essence three main techniques used to obtain asymptotic statements. The first possibility  uses epi-convergence of sample-based approximations of objective functions over compact sets and therefore draws from powerful statements in the theory of variational convergence. The second type of method employs a uniform law of large numbers for sample-based approximations of objective functions. Finally, asymptotic statements can also be derived from stability estimates for optimal values and solutions with respect to probability semimetrics. This requires, amongst other things, that the class of integrands in the objective constitutes a $P$-uniformity class for the semimetric in question.

Given a general stochastic optimization problem
\begin{equation}\label{eq:abs-op}
	\min_{z \in Z_{\rm ad}} \mathbb E_{P}[F(z)],
\end{equation}
an empirical approximation would take the form
\begin{equation}\label{eq:emp-op}
	\min_{z \in Z_{\rm ad}} \mathbb E_{P_N}[F(z)],
\end{equation}
where the original probability measure $P$ is replaced by a (sequence of) typically discrete approximation(s) $P_N$ for $N \in \mathbb N$. For example, the probability measure $P_N$ could be an empirical probability measure associated with a random sample of size $N$ from $P$. This is a common approach often referred to as ``sample average approximation'' (SAA), see e.g., \cite{Kleywegt2002,Shapiro2021}. A data-driven viewpoint can be drawn from machine learning in which \eqref{eq:abs-op} represents the ``population risk minimization'' problem and \eqref{eq:emp-op} the corresponding ``empirical risk minimization'' problem. Here, the underlying probability measure of the data $P$ is typically unknown. It is therefore of interest to understand the behavior of solutions in the big data limit (as $N \to \infty$).

The main questions can be easily stated: Do the optimal values and  solution sets of \eqref{eq:emp-op} converge to their ``true'' counterparts for
\eqref{eq:abs-op} as $N$ passes to infinity and what is the strongest form of stochastic convergence that can be guaranteed? If we treat the $N$-dependent objects as statistical estimators of the true values and seek to prove at least convergence in probability, then these are questions of consistency, cf. \cite{Shapiro2013}.

Motivated by recent advances in
partial differential equation (PDE)-constrained optimization under uncertainty \cite{Garreis2019b,Kouri2018a}, scientific machine learning \cite{Berner2020,Nelsen2021}, nonconvex stochastic programming
\cite{Liu2022,Qi2021,Davis2022}, and statistical estimation \cite{Royset2019,Royset2020,Mei2018}, we provide such consistency results for stochastic optimization problems in which the decision variables $z$ may be taken in an infinite dimensional space $Z$. We will consider more general ``risk-averse'' problems in which the expectation $\mathbb E_{P}$ is allowed to be replaced by certain classes of convex risk functionals $\risk$. And as it is often lacking in the application areas mentioned above, we do not assume convexity of the integrand $F$.  For consistency results on finite dimensional risk-averse stochastic optimization problems, 
we refer the reader to
\cite{Dentcheva2017,Shapiro2013,Shapiro2021}.

From an abstract perspective, we consider stochastic optimization problems of the type
\begin{equation}
	\min_{z \in Z_{\rm ad}} \risk[F(z)] + \wp(z).
\end{equation}
Here, $Z_{\rm ad}$ is typically a closed convex subset of an infinite dimensional space; $\wp$ is a deterministic convex cost function; $F$ is a random integrand that typically depends on the solution of a differential equation subject to random inputs; and $\risk$ is a convex functional that acts as a numerical surrogate for our risk preference, e.g., a convex combination of $\mathbb E_{P}[X]$ and a semideviation 
$\mathbb E_{P}[\max\{0,X - \mathbb E_{P}[X]\}]$.

Despite the past successes in consistency analysis listed above, there is a major difficulty in extending the finite dimensional arguments to the infinite dimensional setting.  In order to use both the epigraphical as well as the uniform law of large numbers approaches, we need an appropriately defined \emph{norm} compact set that contains both the approximate $N$-dependent solutions as well as true solutions. It is not enough for the feasible set to be closed and bounded. For example, the simple set of pointwise bilateral constraints
\[
Z_{\rm ad} \coloneqq \big\{ \, z \in L^2(0,1) \colon \; 0 \le z(x) \le 1 \text{ for a.e. } x \in (0,1) \, \big\}
\]
is weakly sequentially compact in $L^2(0,1)$, but not norm compact. The literature is not void of results for infinite dimensional problems. However, the stability statements developed in \cite{MHoffhues_WRoemisch_TMSurowiec_2021,WRoemisch_TMSurowiec_2021}
and the large deviation-type bounds derived in \cite{Milz2021,Milz2022c}
have only been demonstrated for strongly convex risk-neutral problems. While it may be possible to extend some of these results to a risk-averse setting, it appears rather challenging to obtain statements about the consistency of minimizers without strong convexity.
In the recent preprint \cite{Milz2022}, consistency results
for optimal values and solutions are
established for risk-neutral PDE-constrained optimization
using a uniform law of large numbers. 
Our framework and that in \cite{Milz2022} are different.
Besides considering risk-neutral problems,
i.e., $\Risk{} = \mathbb{E}_P$,  the work \cite{Milz2022}
requires the integrands be continuously differentiable, the decision
space be a separable Hilbert space, and requires a specific strongly convex
control regularization in the objective function. Moreover, 
\cite{Milz2022} establishes consistency for exact solutions while we are able to
establish consistency for approximate solutions using epiconvergence.

The paper is structured as follows. In Section \ref{sec:prelim}, we introduce the basic notation, assumptions, and several preliminary results necessary for the remaining parts of the text. Afterwards, in Section \ref{sec:consist}, we present our main result. Finally, the utility of the main consistency result is demonstrated for several problem classes in Section \ref{sec:apps}.

\section{Notation, Assumptions, and Preliminary Results}\label{sec:prelim}
We introduce several concepts, notation and assumptions that are required in the
text.

\subsection{Probability and Function Spaces}\label{ssec:spaces}
Throughout the text, all spaces are defined over the real numbers $\mathbb R$
and metric spaces are equipped with their Borel $\sigma$-field. Let $\Xi$ be a complete separable metric space, $\mathcal{A}$ the associated Borel $\sigma$-algebra, and $\mathbb P : \mathcal{A} \to [0,1]$ a probability measure. The triple $(\Xi, \mathcal{A}, \mathbb{P})$ is always assumed to be a complete probability
space.
Throughout the manuscript, $(\Omega, \cF, P)$ is a probability space.

If $\Upsilon$ is a  Banach space, then its topological dual space is denoted by $\Upsilon^*$. Their dual pairing is denoted by
$\dualp[\Upsilon]{v}{w}$ for $v \in \Upsilon^*$, $w \in \Upsilon$.
If $\Upsilon$ is reflexive, we identify its bi-dual $(\Upsilon^*)^*$ with $\Upsilon$. Throughout the text, we will use $p \in [1,\infty)$ for a general integrability exponent. In the application section, we will consider problems involving random partial differential equations. These require several function spaces. The underlying physical domain $D \subset \mathbb{R}^d$ with $d \in \{1,2,3\}$ will always be an open bounded Lipschitz domain.

For a Banach space $(V, \left\| \cdot \right\|_V)$
we will denote the  Lebesgue--Bochner space 
$L^p(\Xi, \mathcal{A}, \mathbb{P}; V)$ of all 
strongly $\mathcal{A}$-measurable $V$-valued functions by
\begin{equation*}
	\label{eq:Bochnerraum}
	L^p(\Xi, \mathcal{A}, \mathbb{P}; V) =
	\lbrace   u :  \Xi \rightarrow V : u \text{ strongly 
		$\mathcal{A}$-measurable and }   \left\| u \right\|_{L^p(\Xi, \mathcal{A}, \mathbb{P}; V)} < \infty\rbrace
\end{equation*}
endowed with the natural norms
$
\left\| u \right\|_{L^p(\Xi, \mathcal{A}, \mathbb{P}; V)} = (\mathbb E_{\mathbb{P}}[\| u \|_{V}^p])^{1/p}
$
for $p \in [1,\infty)$
and for bounded fields: 
$\left\| u \right\|_{L^p(\Xi, \mathcal{A}, \mathbb{P}; V)}
= \underset{\xi \in \Xi}{{\mathbb{P}\text{-}\mathrm{ess}} \sup} \left\| u(\xi) \right\|_V$.
In the event that $V = \real$, we simply write 
$L^p(\Xi, \mathcal{A}, \mathbb{P})$. 
For the PDE applications, we use $L^p(D)$ to denote the usual Lebesgue space of $p$-integrable (or essentially bounded) functions over $D$. For more details on Lebesgue--Bochner spaces, we refer the reader to
\cite[Chapter III]{EHille_RSPhillips_1974}. We denote convergence in the norm by $\to$ and weak convergence by $\wto$.
For a sequence $(v_k)$, we denote by $(v_k)_{K}$ a subsequence
of $(v_k)$, where $K \subset \mathbb{N}$ is an infinite index set.

Given two random variables $X_1, X_2 \in \Lp{p}$ for $p \in [1,\infty)$, we say that $X_1$ and $X_2$ are \emph{distributionally equivalent} with respect to $P$ if $ P(X_1 \le t) =  P(X_2 \le t)$ for all $t \in \mathbb R$. A functional $\rho : \Lp{p} \to \mathbb R$ is said to be \emph{law invariant} with respect to $ P$ if for all distributionally equivalent random variables $X_1, X_2 \in  \Lp{p}$ we have $\rho(X_1) = \rho(X_2)$. In this setting, it therefore makes sense to use the (abuse of) notation $\rho(H_X)$, where $H_X(t) \coloneqq P(X \le t)$ with $t\in \mathbb R$ as opposed to $\rho(X)$. We caution that this does not mean we redefine the function $\rho$ over a space of c\`adl\`ag functions.
For a (cumulative) distribution function $H$ defined on $\real$, 
its quantile function $H^{-1}$ is defined by
$H^{-1}(t) \coloneqq \inf_{s\in \real} \{\,s \colon H(s) \geq t\,\}$ for
$t \in (0,1)$.
Let $(\Omega, \cF, P)$ be nonatomic and let
$\rho : \Lp{p} \to \mathbb R$ be law invariant for $p \in [1,\infty)$. Since
$(\Omega, \cF, P)$  is nonatomic, 
there exists a random variable $G : \Omega \to [0,1]$
with uniform distribution $\nu$ on $[0,1]$ \cite[Prop.\ A.7]{Dudley2011}
(see also \cite[Prop.\ 9.1.11]{Bogachev2007}).
Let $X \in \Lp{p}$ be a random variable.
Since $H_X^{-1}(G(\cdot))$ has the same distribution function
as that of $X$ \cite[Prop.\ 9.1.2]{Dudley2002}
and $H_X^{-1}(G(\cdot)) \in L^p(\Omega,\cF, P)$,
we have $\rho(H_{X}) = \rho(H_X^{-1}(G(\cdot)))$.
More generally, we write $\rho(H)$ instead of
$\rho(H^{-1}(G(\cdot)))$, provided that 
$H$ is a distribution function on $\real$ with
$\int_0^1 |H^{-1}(q)|^p \du \nu(q) < \infty$.

\subsection{Convex Analysis and Several Key Functionals}\label{ssec:cvx}
Given a Banach space $V$, the (effective) domain of an extended real-valued function $f : V \to (-\infty,\infty]$, will be denoted by $\dom{f} \coloneqq \{x \in V \colon f(x)< \infty\}$. We typically exclude convex functions that take the value $-\infty$. For $f: V \to (-\infty,\infty]$ and $\varepsilon > 0$, $x_{\varepsilon} \in V$ is an $\varepsilon$-minimizer of $f$ provided $\inf_{v \in V} f(v)$ is \emph{finite} and $f(x_{\varepsilon}) \le \inf_{v \in V} f(v) + \varepsilon$. 
The $\varepsilon$-solution set ($\varepsilon \ge 0$) is then the set
$\mathcal{S}^{\varepsilon} \coloneqq \{x \in V : f(x) \le \inf_{v \in V} f(v) + \varepsilon\}$, 
provided that $\inf_{v \in V} f(v)$ is finite.
We use the convention $\mathcal{S} = \mathcal{S}^{0}$.

Let $\Upsilon$ be a normed space.
For   $x \in \Gamma \subset \Upsilon$
and $\Psi \subset \Upsilon$, we define
\begin{align*}
	\dist{x}{\Psi} = \inf_{y\in \Psi}\, \norm[\Upsilon]{x-y}
	\quad \text{and} \quad
	\deviation{\Gamma}{\Psi} = \sup_{x\in \Gamma}\, 	\dist{x}{\Psi}.
\end{align*}
We recall that a Banach space $V$ has the
\emph{Radon--Riesz (Kadec--Klee) property}
if $v_k \to v$ whenever
$(v_{k}) \subset V$ is a sequence
with $v_{k} \wto v \in V$ and
$\norm[V]{v_{k}} \to \norm[V]{v}$
as $k \to\infty$.
More generally, we will say that a function $\varphi : V \to [0,\infty)$ is an \emph{R-function} if it is convex and
continuous, and if $v_k \to v$ as $k \to\infty$ whenever
$(v_{k}) \subset V$ is a sequence
with $v_{k} \wto v  \in V$ and
$\varphi(v_{k}) \to \varphi(v)$
as $k \to\infty$.
Notions related to  but different from
that of an R-function are available in the literature,
such as functions having the Kadec property and strongly rotund functions
\cite{Borwein1994,Borwein2010}. The notion of an
R-function is first introduced in our manuscript.
If $V$ is a reflexive Banach space, 
then there exists an R-function on $V$ \cite[p.\ 154]{Borwein1994}.
A notion of regularizers different from that of an R-function can be
found in  \cite{Juditsky2022}.

As  the following fact demonstrates, the class of R-functions is rather large
and includes, e.g., typical cost functions and regularizers used in PDE-constrained
optimization. See \Cref{subsect:ocpspde} for an example
of an R-function in the context of PDE-constrained optimization.
\begin{lemma}
	\label{lem:rfunction}
	Let $V$ be a Banach space. If $\wp:[0,\infty) \to [0,\infty)$
	is convex and strictly increasing and $\varphi : V \to [0,\infty)$
	is an R-function, then $\wp \circ \varphi$ is an R-function.
	In particular, if $V$ has the Radon--Riesz property,
	then $\wp \circ \norm[V]{\cdot}$ is an R-function.
\end{lemma}

\begin{proof}
The function $\wp \circ \varphi$
is convex and continuous.
Let $v_k \wto v$ and $\wp(\varphi(v_k)) \to \wp(\varphi(v))$.
Since $\wp$ is strictly increasing
on $[0,\infty)$, it has a continuous inverse.
Hence $\varphi(v_k) \to \varphi(v)$.
\end{proof}

For a Banach space $V$ and a complete probability space
$(\Xi, \mathcal{A}, \mathbb{P})$,
$f: V \times \Xi 
\to (-\infty,\infty]$ is said to be \emph{random lower
	semicontinuous} provided $f$ is jointly measurable
(with respect to the tensor-product $\sigma$-algebra of the
Borel $\sigma$-algebra on $V$ and $\mathcal{A}$) and
$f(\cdot, \xi)$  is lower semicontinuous for every $\xi \in \Xi$.
If $\Upsilon_1$ and $\Upsilon_2$ are  metric spaces,
then $G : \Upsilon_1 \times \Xi \to \Upsilon_2$ is a
\emph{\Caratheodory\ mapping}
provided $G(\upsilon, \cdot)$ is measurable for all $v \in \Upsilon_1$ and $G(\cdot, \xi)$ is continuous
for all $\xi \in \Xi$.

Finally, there are many concepts of risk measures in the literature. We will work with the following with further refinements as needed in the text below.
Let	$\rho : \Lp{p} \to (-\infty,\infty]$.
We consider the following conditions on the functional $\rho$.
\begin{enumerate}[nosep,label=(R{\arabic*})]
	\item Convexity. For all $X,Y \in \Lp{p} $ and $\lambda \in (0,1)$, we have $\rho(\lambda X + (1-\lambda) Y) \le \lambda \rho(X) + (1-\lambda)\rho(Y)$.
	\item Monotonicity. For all $X,Y \in \Lp{p} $ such that $X \le Y$ \wpone, we have $\rho(X) \le \rho (Y)$.
	\item Translation equivariance. If $X \in \Lp{p} $ and $C$ is a degenerate random variable with $C = c$ \wpone for some $c \in \mathbb R$, then $\rho(X + C) = \rho(X) + c$.
	\item Positive homogeneity. If $X \in \Lp{p} $ and $\gamma > 0$, then $\rho(\gamma X) = \gamma \rho(X)$.
\end{enumerate}
The risk measure $\rho$ is called \emph{convex}
if it satisfies (R1)--(R3)
and it is referred to
as \emph{coherent} if it satisfies (R1)--(R4), see \cite{PArtzner_FDelbaen_JMEber_DHeath_1999a},
\cite{HFollmer_ASchied_2002a}, and in particular
\cite[p.\ 231]{Shapiro2021}.

\subsection{Epiconvergence and Weak Inf-Compactness}
Variational convergence, in particular (Mosco-)epiconvergence, plays a central role in consistency analysis. We provide here the necessary definitions and results from the literature. In addition, we prove several new results that are tailored to applications involving PDEs with random inputs.

We recall the notions of epiconvergence and Mosco-epiconvergence
\cite{Attouch1989,Dontchev1993}.

\begin{definition}[Epiconvergence]
	Let $V$ be a complete metric space. Let $\phi_k \colon V \to (-\infty,\infty]$
	be a sequence and let $\phi: V \to (-\infty,\infty]$ be a function.
	The sequence $(\phi_k)$ \emph{epiconverges} to $\phi$ if for each
	$v \in V$
	\begin{enumerate}[nosep]
		\item and each $(v_k) \subset V$
		with $v_k \to v$ as $k \to\infty$,
		$\liminf_{k \to \infty} \, \phi_k(v_k) \geq \phi(v)$, and
		\item there exists $(v_k) \subset V$
		with $v_k \to v$ as $k \to\infty$
		such that $\limsup_{k \to \infty} \, \phi_k(v_k) \leq \phi(v)$.
	\end{enumerate}
\end{definition}

In many instances in infinite dimensional optimization, especially the calculus of variations, optimal control, and PDE-constrained optimization we are forced to work with weaker topologies in the context of variational convergence. If the underlying space is a reflexive Banach space, then we may appeal to epiconvergence in the sense of Mosco, which was introduced in \cite{UMosco_1969}.

\begin{definition}[Mosco-epiconvergence]
	Let $V$ be a reflexive Banach
	space and let $ V_0 \subset V$ be a closed, nonempty, convex set.
	Let $\phi_k \colon V_0 \to (-\infty,\infty]$
	be a sequence and let $\phi: V_0 \to (-\infty,\infty]$ be a function.
	The sequence $(\phi_k)$
	\emph{Mosco-epiconverges} to $\phi$ if for each
	$v \in V_0$
	\begin{enumerate}[nosep]
		\item and each  $(v_k) \subset V_0$
		with $v_k \wto v$ as $k \to\infty$,
		$\liminf_{k \to \infty} \, \phi_k(v_k) \geq \phi(v)$, and
		\item there exists $(v_k) \subset V_0$
		with $v_k \to v $
		such that $\limsup_{k \to \infty} \, \phi_k(v_k) \leq \phi(v)$.
	\end{enumerate}
\end{definition}

In the definition of Mosco-epiconvergence, we allow
for the sequence $(\phi_k)$
and the epi-limit $\phi$ to be defined on a nonempty, convex, closed subset of
a reflexive Banach space. This allows us to model constraints without
the need for indicator functions. We will see below in \Cref{lem:thm53}
that this variation on the original definition leaves the crucial implications
of Mosco-epiconvergence intact. In other words,
\Cref{lem:thm53} provides conditions sufficient
for consistency of optimal values of Mosco-epiconvergent objective functions;
compare with \cite[Thm.\ 1.10]{Attouch1984}, \cite[Thm.\ 5.3]{Dong2000},
and \cite[Thm.\ 6.2.8]{Borwein2010}, for example.
\begin{theorem}
	\label{lem:thm53}
	Let $V$ be a reflexive Banach space  and let $V_0 \subset V$
	be a closed, nonempty, convex set.
	Suppose that $\fun_k \colon V_0 \to (-\infty,\infty]$ Mosco-epiconverges to
	$\fun \colon V_0 \to (-\infty,\infty]$.
	Let $(v_{k}) \subset V_0$
	and $(\varepsilon_k) \subset [0,\infty)$
	be sequences such that $\varepsilon_k \to 0^+$
	and for each $k \in \mathbb{N}$,
	let $v_{k}$ satisfy
	\[
	\fun_k(v_{k}) \leq \inf_{v \in V_0}\, \fun_k(v) + \varepsilon_k.
	\]
	If $(v_{k})_K$ is a subsequence
	of $(v_{k})$ such that $v_{k} \wto \bar v$
	as $K \ni k \to\infty$,
	then
	\begin{enumerate}[nosep]
		\item $\bar v \in  V_0$,
		\item $\fun(\bar v) = \inf_{v \in V_0}\, \fun(v)$,
		\item $\inf_{v \in  V_0}\, \fun_k(v) \to \inf_{v \in
			V_0}\, \fun(v)$ as $K \ni k \to\infty$,
		\item $\fun_k(v_{k}) \to \fun(\bar v)$ as $K \ni k \to\infty$.
	\end{enumerate}
\end{theorem}%
\begin{proof}
Since $(v_{k}) \subset V_0$ and $ V_0$
is weakly sequentially closed,  we have $\bar v \in V_0$.
Since $(\fun_k)$ Mosco-epiconverges to $\fun$ on $ V_0$,
it epiconverges to $\fun$, where $V_0$ may be understood
as a complete metric space using the norm topology. Hence
\begin{align}
	\label{eq:limsuphk}
	\limsup_{k \to \infty}\, \inf_{v \in V_0}\, \fun_k(v)
	\leq \inf_{v \in V_0}\, \fun(v);
\end{align}
see, e.g., \cite[Props.\ 1.14 and 2.9]{Attouch1984}.
Then Mosco-epiconvergence ensures
\begin{align*}
	\liminf_{K \ni k \to \infty}\,
	\inf_{v \in V_0}\, \fun_k(v)
	=
	\liminf_{K \ni k \to \infty}\,
	[\varepsilon_k + \inf_{v \in V_0}\, \fun_k(v)]
	\geq
	\liminf_{K \ni k \to \infty}\, \fun_k(v_{k})
	\geq \fun(\bar v).
\end{align*}
Combined with \eqref{eq:limsuphk}, we find that
$\fun(\bar v) = \inf_{v \in V_0}\, \fun(v)$
and $\inf_{v \in V_0}\, \fun_k(v) \to \inf_{v \in V_0}\, \fun(v)$
as $K \ni k \to\infty$.
The assertion $\fun_k(v_{k}) \to \fun(\bar v)$ as $K \ni k \to\infty$
is implied by the above derivations and
\begin{align*}
	\limsup_{K \ni k \to \infty}\, \fun_k(v_{k})
	\leq
	\limsup_{K \ni k \to \infty}\,  [\varepsilon_k+
	\inf_{v \in V_0}\, \fun_k(v)]
	=
	\limsup_{K \ni k \to \infty}\,
	\inf_{v \in V_0}\, \fun_k(v)
	\leq
	\limsup_{k \to \infty}\,
	\inf_{v \in V_0}\, \fun_k(v).
\end{align*}
\end{proof}

\Cref{lem:solutionsetcompact} demonstrates a weak compactness property
of approximate minimizers to ``regularized'' optimization problems
with Mosco-epiconvergent objective functions.
Let $\csp$ be a reflexive Banach space,
let $\adcsp \subset \csp$ be a closed, nonempty, convex set, and
let $f_k$, $f : \adcsp \to (-\infty,\infty]$.
Furthermore, let $\varphi : \csp \to [0,\infty)$ be a convex, continuous
function. 	We define the optimal values
\begin{align}
	\label{eq:epiregularization}
	\mathfrak{m}_k^*  \coloneqq \inf_{z\in \adcsp}\, \{ \, f_k(z)
	+ \varphi(z) \, \}
	\quad \tand \quad
	\mathfrak{m}^* \coloneqq
	\inf_{z\in \adcsp}\, \{ \, f(z)
	+ \varphi(z) \,\}
\end{align}
and the solution sets
\begin{align*}
	\mathcal{S}_k^{\varepsilon_k}  \coloneqq \{\, z \in \adcsp\colon
	f_k(z) + \varphi(z) \leq  \mathfrak{m}_k^* +\varepsilon_k \, \}
	\quad \tand \quad
	\mathcal{S} \coloneqq \{\, z \in \adcsp \colon
	f(z) + \varphi(z) = \mathfrak{m}^* \, \}.
\end{align*}

\begin{proposition}
	\label{lem:solutionsetcompact}
	Let  $\csp$ be a reflexive Banach space,
	let $\adcsp \subset \csp$  be a nonempty, closed,  convex set,
	let $\varphi : \csp \to [0,\infty)$
	be a convex, continuous function,
	and let $\csp_0  \subset \csp$ be  bounded.
	Suppose that $f_k  \colon \adcsp \to (-\infty,\infty]$
	Mosco-epiconverges
	to $f : \adcsp \to (-\infty,\infty]$.
	Let $(\varepsilon_k) \subset [0,\infty)$
	be a sequence with $\varepsilon_k \to 0^+$.
	Suppose that $\mathcal{S} \not= \emptyset$
	and that for all $k \in \mathbb{N}$,
	\begin{align*}
		\mathcal{S}_k^{\varepsilon_k} \subset \csp_0
		\quad \tand \quad
		\mathcal{S}_k^{\varepsilon_k} \not= \emptyset.
	\end{align*}
	If $(z_{k})$ is a sequence with
	$z_{k} \in \mathcal{S}_k^{\varepsilon_k}$
	for all $k \in \mathbb{N}$ and $(z_{k})_K$ is a subsequence
	of $(z_{k})$, then $(z_{k})_K$ has a further subsequence
	$(z_{k})_{K_1}$
	converging \emph{weakly}  to some $\bar  z \in \mathcal{S}$
	and $\varphi(z_{k}) \to \varphi(\bar z)$ as $K_1 \ni k \to\infty$.
\end{proposition}
\begin{proof}
Since $(z_{k})_K \subset \adcsp$,
$(z_{k})_K \subset \csp_0$,
$\csp_0$ is bounded, and $\adcsp$ is closed and convex,
$(z_{k})_K$ has a further subsequence
$(z_{k})_{K_1}$ such that $z_{k} \wto \bar z \in \adcsp$
as $K_1 \ni k \to\infty$
\cite[Thms.\ 2.23 and 2.28]{Bonnans2000}.
Since $\bar z \in \adcsp$,
the Mosco-epiconvergence of $(f_k)$ to $f$
ensures the
existence of a sequence $(\tilde z_{k}) \subset \adcsp$
such that $\tilde z_{k} \to \bar z \in \adcsp$
as $k \to\infty$
and $\limsup_{k \to \infty} \,f_k(\tilde z_{k}) \leq f(\bar z)$.
Since $\tilde z_{k} \to \bar z$ implies
$\tilde z_{k} \wto \bar z$, we have
$\lim_{k \to \infty} \,f_k(\tilde z_{k}) = f(\bar z)$.
Since 
$z_{k} \in \mathcal{S}_k^{\varepsilon_k}$
and $\tilde z_{k} \in \adcsp$, we have
for all $k \in \mathbb{N}$,
\begin{align}
	\label{eq:fkzkfktildezk}
	f_k(z_{k}) + \varphi(z_{k})
	\leq f_k(\tilde z_{k}) + \varphi(\tilde z_{k})
	+ \varepsilon_k.
\end{align}
Since $(f_k)$ Mosco-epiconverges to $f$,
we have
$
f(\bar z) \leq \liminf_{K_1 \ni k \to \infty}\, f_k(z_{k})
$.
Combined with the fact that $\varphi$ is continuous and
\begin{align*}
	\liminf_{K_1 \ni k \to \infty}\, f_k(z_{k})
	+ \limsup_{K_1 \ni k \to \infty}\, \varphi(z_{k})
	\leq \limsup_{K_1 \ni k \to \infty}\,
	f_k(z_{k}) + \varphi(z_{k}),
\end{align*}
the estimate \eqref{eq:fkzkfktildezk} ensures
\begin{align*}
	f(\bar z) +
	\limsup_{K_1 \ni k \to \infty}\,
	\varphi(z_{k})
	&\leq
	\limsup_{K_1 \ni k \to \infty}\,
	f_k(\tilde z_{k}) + \varphi(\tilde z_{k})
	+ \varepsilon_k
	\leq
	\limsup_{k \to \infty}\,
	f_k(\tilde z_{k}) + \varphi(\tilde z_{k})
	+ \varepsilon_k
	\\
	& =
	\lim_{k \to \infty}\,
	f_k(\tilde z_{k}) + \varphi(\tilde z_{k})
	+ \varepsilon_k
	=
	f(\bar z) + \varphi(\bar z).
\end{align*}
Since $z_{k} \wto \bar z$
as $K_1 \ni k \to\infty$, $\mathcal{S} \neq \emptyset$,
and $(f_k)$ Mosco-epiconverges to $f$,
\Cref{lem:thm53} ensures $\bar z \in \mathcal{S} $.
Since $\bar z \in \mathcal{S}$,
we have $f(\bar z) \in \real$.
Thus
$\limsup_{K_1 \ni k \to \infty} \varphi(z_{k}) \leq
\varphi(\bar z)$.
Since $\varphi$ is convex and continuous,
it is weakly lower semicontinuous.
Combined with
$z_{k} \wto \bar z$ as $K_1 \ni k \to\infty$, we
have $\varphi(z_{k}) \to \varphi(\bar z)$
as $K_1 \ni k \to\infty$.
\end{proof} 

While the sum of an Mosco-epiconvergent sequence
and a convex, continuous function Mosco-epiconverge, \Cref{lem:solutionsetcompact}
allows us to  draw further conclusions about the minimizers to
composite optimization problems defined by sums of  Mosco-epiconvergent
and convex, continuous functions than a direct application of
the ``sum rule.'' For example, if $\varphi$ is an R-function, then the sequence
$(z_{k})_{K_1}$ considered in
\Cref{lem:solutionsetcompact} converges strongly to
an element of $\mathcal{S}$.
\begin{corollary}
	\label{lem:solutionsetstronglycompact}
	If the hypotheses of \Cref{lem:solutionsetcompact} hold true and
	$\varphi$ is an R-function,
	then each subsequence of $(z_{k})$ has a further subsequence
	converging \emph{strongly} to an element of $\mathcal{S}$.
\end{corollary}
\begin{remark}
	\normalfont
	If $\adcsp$ is bounded, then we can choose $\csp_0  = \adcsp$
	in \Cref{lem:solutionsetcompact}.
	The condition $\mathcal{S}_k^{\varepsilon_k} \subset \csp_0$
	for all $k \in \natural$
	in \Cref{lem:solutionsetcompact} is related to a
	``weak 	inf-compactness'' condition,
	provided that $\csp_0$ is also convex and bounded.
	In this case, $\csp_0$ is weakly (sequentially) compact.
	Instead of requiring
	$\mathcal{S}_k^{\varepsilon_k} \subset \csp_0$ 	for all $k \in \natural$,
	we could require
	for some $\gamma \in \real$
	and for all $k \in \mathbb{N}$,
	\begin{align}
		\label{eq:levelset}
		\emptyset \neq \{\, z \in \adcsp\colon \,
		f_k(z) + \varphi(z) \leq \gamma
		\,\}
		\subset \csp_0.
	\end{align}
	The level set condition \eqref{eq:levelset}
	ensures that $\mathcal{S}_k^{\varepsilon_k}$ is nonempty, provided
	that $f_k$ is weakly lower semicontinuous.
	In case that $\csp_0$ is norm compact,
	the condition \eqref{eq:levelset} has been used, for example,
	in Theorem~2.1 in \cite{Lachout2005} to establish
	consistency properties for infinite dimensional stochastic programs.
	If $\sup_{k\in \mathbb{N}}\, \mathfrak{m}_k^* < \infty$,
	$\gamma > \sup_{k\in \mathbb{N}}\, \mathfrak{m}_k^*$ and
	for all $k \in \mathbb{N}$,
	\begin{align*}
		\{\, z \in \adcsp \colon \,
		f_k(z) + \varphi(z) \leq \gamma
		\,\}
		\subset \csp_0,
	\end{align*}
	then $\mathcal{S}_k^{\varepsilon_k} \subset \csp_0$
	for all sufficiently large $k \in \natural$
	since we eventually have
	$\sup_{k\in \mathbb{N}}\, \mathfrak{m}_k^* + \varepsilon_k \leq \gamma$.
\end{remark}

\begin{corollary}
	\label{cor:deterministicconsistency}
	If the hypotheses of \Cref{lem:solutionsetcompact} hold,
	then
	$\mathfrak{m}_k^* \to \mathfrak{m}^*$
	as $k \to\infty$.
	If furthermore $\varphi$ is an R-function,
	then
	$\deviation{\mathcal{S}_k^{\varepsilon_k}}{\mathcal{S}} \to 0$
	as $k \to\infty$.
\end{corollary}
\begin{proof}
Let	$z_{k} \in \mathcal{S}_k^{\varepsilon_k}$ for each
$k \in \mathbb{N}$.
The hypotheses ensure that
$(f_k)$ Mosco-epiconverges to $f$.
Let $(\mathfrak{m}_k^*)_K$ be a subsequence of
$(\mathfrak{m}_k^*)$.
\Cref{lem:solutionsetcompact}
ensures that $(z_{k})_K$ has a further subsequence
$(z_{k})_{K_1}$ that weakly converges to some element in
$\mathcal{S}$.
Combined with \Cref{lem:thm53}, we
find that $\mathfrak{m}_k^* \to \mathfrak{m}^*$
as $K_1 \ni k \to\infty$.
Since  $\mathcal{S}$ is nonempty,
$\mathfrak{m}^* \in \real$.
Putting together the pieces, we have shown that
each subsequence of $(\mathfrak{m}_k^*)$ has a further subsequence
converging to $\mathfrak{m}^*$. Hence $\mathfrak{m}_k^* \to \mathfrak{m}^*$
as $k \to\infty$.

It must still be shown that
$\deviation{\mathcal{S}_k^{\varepsilon_k}}{\mathcal{S}} \to 0$
as $k \to\infty$.
Since $\mathcal{S}_k^{\varepsilon_k} \subset \csp_0$
and
$\mathcal{S} \subset \csp$
are nonempty, and $\csp_0$ is bounded, we have
$\deviation{\mathcal{S}_k^{\varepsilon_k}}{\mathcal{S}}
\leq \deviation{\csp_0}{\mathcal{S}} <\infty
$
for all $k \in \mathbb{N}$.
Let us define the sequence
$\varrho_k \coloneqq
\deviation{\mathcal{S}_k^{\varepsilon_k}}{\mathcal{S}}$.
Let $(\varrho_k)_K$ be a subsequence of $(\varrho_k)$.
We have just shown that
$\varrho_k \leq \deviation{\csp_0}{\mathcal{S}} <\infty$.
Moreover $\varrho_k \geq 0$.
Using the definition of the deviation,
we find that there exists
for each $k \in \mathbb{N}$,
$\tilde z_{k} \in \mathcal{S}_k^{\varepsilon_k}$
such that $\varrho_k \leq \dist{\tilde z_{k}}{\mathcal{S}} + 1/k$.
\Cref{lem:solutionsetstronglycompact}
ensures that $(\tilde z_{k})_K$ has a further subsequence
$(\tilde z_{k})_{K_1}$ that
strongly converges to some $\bar z \in \mathcal{S}$.
Since $\mathcal{S}$ is nonempty, $\dist{\cdot}{\mathcal{S}}$
is (Lipschitz) continuous \cite[Thm.\ 3.16]{Aliprantis2006}.
It follows that
$\dist{\tilde z_{k}}{\mathcal{S}} \to 0$ as $K_1 \ni k \to\infty$.
Hence
$\varrho_k \to 0$ as $K_1 \ni k \to\infty$.
Since each subsequence of $(\varrho_k)$ has a further subsequence
converging to zero, $\varrho_k \to 0$ as $k \to\infty$.
\end{proof} 

\Cref{prop:moscoepiconvergence} demonstrates that epiconvergence
can imply Mosco-epiconvergence. 
This result is particularly relevant for
PDE-constrained problems. 
If $V$ is a Banach space and $Y$ is a complete metric space,
we refer to a mapping $G : V \to Y$ as \emph{completely continuous}
if $(v_{k}) \subset V$ and $v_{k} \wto v \in V$ implies
$G(v_{k}) \to G(v)$.

\begin{proposition}
	\label{prop:moscoepiconvergence}
	Let $\csp_0 \subset \csp$ be a nonempty, closed, convex subset
	of a  reflexive Banach space $\csp$  and
	let $Y_0 \subset Y$ be a closed subset of a Banach space $Y$.
	Suppose that $\B : \csp \to Y$ is linear and completely continuous
	with $\B(\csp_0) \subset Y_0$.
	If $\fun_k: Y_0 \to (-\infty,\infty]$ epiconverges to
	$\fun: Y_0 \to (-\infty,\infty]$
	and
	$\fun_k \circ \B \colon \csp_0 \to (-\infty,\infty]$
	epiconverges to
	$\fun \circ \B \colon \csp_0 \to (-\infty,\infty]$,
	then
	$\fun_k \circ \B \colon \csp_0 \to (-\infty,\infty]$ Mosco-epiconverges to
	$\fun \circ \B \colon \csp_0\to (-\infty,\infty]$.
\end{proposition}
\begin{proof}
Fix $\bar z \in \csp_0$.
Let $(z_{k}) \subset \csp_0$ be a sequence
with $z_{k} \wto \bar z$.
We have
$\B z_{k}$, $\B \bar z \in Y_0$.
The complete continuity of $\B$ yields
$\B z_{k} \to \B  \bar z$ as $ k \to\infty$.
Since $(\fun_k)$ epiconverges to $\fun$,
$\liminf_{k \to \infty}\, \fun_k(\B z_{k}) \geq \fun(\B \bar z)$.
The hypotheses ensure that $\fun_k \circ \B$
epiconverges to $\fun \circ \B$.
Putting together the pieces, we conclude that
$(\fun_k \circ \B)$ Mosco-epiconverges to 
$\fun  \circ \B$.
\end{proof}

For the applications considered in
\Cref{subsect:ocpspde,subsect:ocpsvi}, $\B$ is the adjoint operator
of a compact (Sobolev) embedding operator, and hence linear and completely
continuous.

\section{Consistency of Empirical Approximations}\label{sec:consist}

We consider the potentially infinite dimensional risk-averse stochastic
program
\begin{equation}\label{eq:gen-risk-op}
	\min_{z \in Z_{\rm ad}} \risk[F(\B z)] + \wp(z),
\end{equation}
where
\begin{align}
	\label{eq:def:Fyomega}
	F(y)(\omega) \coloneqq \pobj(y,\xi(\omega)),
\end{align}
the set $Y_0$ is a closed subset of a separable
Banach space $Y$, $f : Y_0 \times \Xi \to \real$ is a \Caratheodory\ function,
and $\B : \csp \to Y$ is a linear, continuous operator.
Moreover $\xi \colon \Omega \to \Xi$ is a random element
with law  $\mathbb{P} = P \circ \xi^{-1}$,
$f(y, \cdot) \in L^p(\Xi, \cA, \mathbb{P})$, 
and $\mathcal{R} \colon L^p(\Omega, \cF, P) \to \real$
with $1 \leq p  < \infty$ is a law invariant convex risk measure.

We introduce the empirical approximation of \eqref{eq:gen-risk-op}
for the case when $(\Omega, \cF, P)$ is nonatomic.
Let $\xi^1, \xi^2, \ldots$ be defined on 
a complete probability space $(\Omega^\prime,\cF^\prime,P^\prime)$
and assume the sequence is composed of
independent identically distributed $\Xi$-valued random elements
each with law $\mathbb{P} = P \circ \xi^{-1}$.
For $y \in Y_0$ and $N \in \mathbb{N}$, the 
empirical distribution function $\hat{H}_{y,N}(\cdot; \omega^\prime)$ 
of the sample $\pobj(y, \xi^1(\omega^\prime)), \ldots, \pobj(y, \xi^N(\omega^\prime))$
is defined by 
$$\hat{H}_{y,N}(t; \omega^\prime) \coloneqq 
\frac{1}{N} \sum_{i=1}^N \mathbf{1}_{(-\infty, t]}(\pobj(y, \xi^i(\omega^\prime))),$$ 
where $\omega^\prime \in \Omega^\prime$
and $\mathbf{1}_{(-\infty, t]}$ 
is the indicator function of the interval $(-\infty, t]$.
We denote by $\hat{H}_{y,N}^{-1}(\cdot; \omega^\prime)$
the quantile function of $\hat{H}_{y,N}(\cdot; \omega^\prime)$.
We often omit writing the second arguments of $\hat{H}_{y,N}$
and $\hat{H}_{y,N}^{-1}$.
The empirical approximation of \eqref{eq:gen-risk-op} is given by
\begin{equation}\label{eq:emp-risk-op}
	\min_{z \in Z_{\rm ad}} \risk[\hat{H}_{\B z,N}]
	+ \wp(z).
\end{equation}
Recall from our discussion on law invariant risk measures
that $\risk[\hat{H}_{\B z,N}]$ means the risk measure $\risk$
does not distinguish between $z,N$-dependent random variables with distribution
functions equivalent to $\hat{H}_{\B z,N}$.

Our consistency analysis is based on the conditions in 
\Cref{assumption:strongconsistency}. 

\begin{assumption}
	\label{assumption:strongconsistency}
	\normalfont
	\begin{enumerate}[nosep]
		\item The space $\csp$ is a separable, reflexive Banach space,
		$\adcsp \subset \csp$ is nonempty, closed, convex and bounded. 
		The space $Y$ is a separable Banach space,
		and $Y_0 \subset Y$ is closed,
		and $p \in [1,\infty)$.
		\item The mapping $\B : \csp \to Y$ is linear and completely
		continuous
		and $\B(\adcsp) \subset Y_0$.
		\item  The function
		$\wp \colon \csp \to [0,\infty)$ is convex and continuous.
		\item The function $\pobj : Y_0 \times \Xi \to \real$
		is a \Caratheodory\ function.
		\item For all $y \in Y_0$,
		$\pobj(y, \cdot) \in L^p(\Xi, \mathcal{A}, \mathbb{P})$
		and $F : Y_0 \to L^p(\Omega,\cF, P)$
		defined in \eqref{eq:def:Fyomega}
		is continuous.
		\item For each $\bar y \in Y_0$, there exists a neighborhood
		$\mathcal{Y}_{\bar y} \subset Y_0$ of $\bar y$
		and a random variable $h \in L^p(\Xi, \mathcal{A}, \mathbb{P})$
		such that $\pobj(y, \cdot) \geq h(\cdot)$
		for all $y \in \mathcal{V}_{\bar y}$.
	\end{enumerate}
\end{assumption}

Let $m^*$ be the optimal value of problem \eqref{eq:gen-risk-op}
and let $\mathscr{S}$ be its solution set.
Furthermore, let $\hat{m}_N^*$ be the optimal value of \eqref{eq:emp-risk-op}
and let $\hat{\mathscr{S}}_N^r$ be its set of $r$-minimizers,
where $r \geq 0$.
The ``with probability one''-statements in 
\Cref{thm:consistency} are with respect to $P^\prime$.

\begin{theorem}
	\label{thm:consistency}
	Let \Cref{assumption:strongconsistency} hold.
	Suppose further that $(\Omega,\cF, P)$ is nonatomic and complete.
	Let  $\Risk{} : L^p(\Omega,\cF, P) \to \real$ be
	a convex, law invariant  risk measure.
	If $(r_N) \subset [0,\infty)$ is a deterministic sequence
	such that $r_N \to 0$ as $N \to\infty$,
	then $\hat{m}_N^* \to m^*$
	\wpone as $N \to\infty$.
	If furthermore $\wp$ is an R-function, then
	$\deviation{\hat{\mathscr{S}}_N^{r_N}}{\mathscr{S}} \to 0$
	\wpone as $N \to\infty$.
\end{theorem}

To establish \Cref{thm:consistency},
we verify the hypotheses of
\Cref{cor:deterministicconsistency,cor:lln31} (in the Appendix).
\begin{lemma}
	\label{prop:RBzwlsc}
	If \Cref{assumption:strongconsistency} holds and
	$\Risk{} : L^p(\Omega,\cF, P) \to \real$  is a convex risk measure,
	then $\adcsp \ni z \mapsto \risk[F(\B z)]$ is
	completely continuous.
\end{lemma}
\begin{proof}
Since $\Risk{}$ is a finite-valued convex risk measure,
it is continuous \cite[Cor.\ 3.1]{Ruszczynski2006b}.
\Cref{assumption:strongconsistency} ensures the continuity
of $y \mapsto F(y)$. Hence $ y \mapsto \risk[F(y)]$
is continuous.
Now the complete continuity of $\B$ implies that of
$z \mapsto \risk[F(\B z)]$.
\end{proof} 

We recall from our discussion on law invariant risk measures
in  \Cref{sec:prelim}, the identity 
$\risk[\hat{H}_{y,N}(\omega^\prime)]=\risk[\hat{H}_{y,N}^{-1}(G(\cdot);\omega^\prime)]$,
where $G \colon \Omega \to [0,1]$ is a random variable with uniform
distribution $\nu$.

\begin{lemma}
	\label{prop:emp-Rylsc}
	Let \Cref{assumption:strongconsistency} hold.
	Suppose further that $(\Omega,\cF, P)$ is nonatomic and complete.
	Let  $\Risk{} : L^p(\Omega,\cF, P) \to \real$ be a
	convex, law invariant risk measure.
	Then 
	$Y_{0} \times \Omega^\prime \ni (y,\omega^\prime) \mapsto
	\risk[\hat{H}_{y,N}^{-1}(G(\cdot); \omega^\prime)]$
	is a \Caratheodory\ function.
\end{lemma}
\begin{proof}
Let
$\pobj(y,\xi^{(1)}) \leq \cdots \leq \pobj(y,\xi^{(N)})$
be the order statistics of the  sample
$\pobj(y,\xi^1), \ldots, \pobj(y,\xi^N)$.
For $q \in (0,1]$ and $y \in Y_0$, we have
$\hat{H}_{y,N}(q; \omega^\prime) = \pobj(y,\xi^{(j)}(\omega^\prime))$
if $q \in ((j-1)/N, j/N]$ irrespective of whether the sample is distinct. 

We show that $y \mapsto \risk[\hat{H}_{y,N}^{-1}(G(\cdot); \omega^\prime)]$ 
is continuous
for each $\omega^\prime \in \Omega^\prime$.
Let $y_k \to y$ and fix $\omega^\prime \in \Omega^\prime$.
Using the fact that $\nu$ is the uniform distribution
and $P \circ G^{-1} = \nu$, we have
\begin{align*}
	\int_{\Omega}
	|\hat{H}_{y_k,N}^{-1}(G(\omega); \omega^\prime)-\hat{H}_{y,N}^{-1}(G(\omega); \omega^\prime)|^p
	\du P(\omega)
	&=
	\int_{0}^1
	|\hat{H}_{y_k,N}^{-1}(q; \omega^\prime)-\hat{H}_{y,N}^{-1}(q; \omega^\prime)|^p
	\du \nu(q)
	\\
	&=
	\frac{1}{N} \sum_{i=1}^N
	|\pobj(y_k,\xi^i(\omega^\prime))-\pobj(y,\xi^i(\omega^\prime))|^p.
\end{align*}
Since $\pobj$ is a \Caratheodory\ function and $p \in [1,\infty)$,
it follows that
$\hat{H}_{y_k,N}^{-1}(G(\cdot); \omega^\prime) \to \hat{H}_{y,N}^{-1}(G(\cdot); \omega^\prime)$
in $L^p(\Omega,\cF, P)$.
Combined with the continuity of $\Risk{}$, we have
$\Risk{\hat{H}_{y_k,N}^{-1}(G(\cdot); \omega^\prime)} \to 
\Risk{\hat{H}_{y,N}^{-1}(G(\cdot); \omega^\prime)}$
as $k \to \infty$. Consequently,
$y \mapsto \Risk{\hat{H}_{y,N}^{-1}(G(\cdot),\omega^\prime)}$ is continuous
for each $\omega^\prime \in \Omega^\prime$.

For each fixed $y \in Y_0$, the function
$\omega^\prime \mapsto \hat{H}_{y,N}^{-1}(G(\cdot); \omega^\prime) \in L^p(\Omega,\cF, P)$
is measurable because it is the composition of a piecewise
constant and measurable functions.

Combining these arguments, we find that
$(y,\omega^\prime) \mapsto \risk[\hat{H}_{y,N}^{-1}(G(\cdot); \omega^\prime)]$
is a \Caratheodory\ function.
\end{proof}

\begin{corollary}
	\label{lem:emp-risk-op-has-solution}
	Under the hypotheses of \Cref{prop:emp-Rylsc},
	\textnormal{(a)}
	$\mathscr{S}$ is nonempty and closed,
	\textnormal{(b)}
	$\hat{\mathscr{S}}_N^{r}$ has nonempty, closed images
	for each $r \in [0,\infty)$,
	and
	\textnormal{(c)}
	$\hat{m}_N^*$ 	and	$\hat{\mathscr{S}}_N^{r}$ are measurable
	for each $r \in [0,\infty)$.
\end{corollary}
\begin{proof}
\begin{enumerate}[wide,nosep]
	\item[(a)] 	Since the set $\adcsp$ is nonempty, closed, convex,
	and bounded,
	\Cref{prop:RBzwlsc} when combined with the direct method of the
	calculus of variations ensures the assertions.
	\item[(b)] Using the properties of $\adcsp$ listed in
	part~(a), \Cref{prop:emp-Rylsc} when combined with
	the direct method of the calculus of variations
	and the complete continuity of $\B$ ensures the assertions.
	\item[(c)]
	Since $\B$ is completely continuous
	and $\csp$ is a Banach space, $\B$ is continuous.
	\Cref{prop:emp-Rylsc}, the continuity of
	$\B$, and Theorem~8.2.11 in \cite{Aubin2009} imply the measurability
	assertions.
\end{enumerate}
\end{proof} 

\begin{proof}[{Proof of \Cref{thm:consistency}}]
To establish the consistency statements, 	we verify the hypotheses of
\Cref{cor:deterministicconsistency,cor:lln31}.
\Cref{lem:emp-risk-op-has-solution} ensures that
$\mathscr{S}$ is nonempty. Hence
$\dist{\cdot}{\mathscr{S}}$ is (Lipschitz)
continuous \cite[Thm.\ 3.16]{Aliprantis2006}.
\Cref{lem:emp-risk-op-has-solution} implies that $\hat{m}_N^*$ is measurable
and that $\hat{\mathscr{S}}_N^{r_N}$ is measurable
with closed, nonempty images.
Combined with Theorem~8.2.11 in \cite{Aubin2009},
it follows that $\deviation{\hat{\mathscr{S}}_N^{r_N}}{\mathscr{S}}$ is measurable.

\Cref{cor:lln31} ensures that
for almost all $\omega^\prime \in \Omega^\prime$,
$\adcsp \ni z \mapsto \risk[\hat{H}_{\B z,N}^{-1}(G(\cdot); \omega^\prime)]$
Mosco-epiconverges
to $\adcsp \ni z \mapsto \risk[F(\B z)]$ as $N \to\infty$.
We have $\hat{\mathscr{S}}_N^{r_N}\subset \adcsp$.
Moreover, $\hat{\mathscr{S}}_N^{r_N}$	 and $\mathscr{S} \subset \adcsp$
are nonempty, and $\wp$ is continuous and convex.
Now, for almost all $\omega^\prime \in \Omega^\prime$,
\Cref{cor:deterministicconsistency} ensures that
$\hat{m}_N^*(\omega^\prime)\to m^*$
as $N \to\infty$.
Hence \wpone, $\hat{m}_N^*\to m^*$ as $N \to\infty$.
If furthermore $\wp$ is an R-function, then
for almost all $\omega^\prime \in \Omega^\prime$,
\Cref{cor:deterministicconsistency} ensures
$\deviation{\hat{\mathscr{S}}_N^{r_N}(\omega^\prime)}{\mathscr{S}} \to 0$
as $N \to\infty$.
Hence \wpone,
$\deviation{\hat{\mathscr{S}}_N^{r_N}}{\mathscr{S}}\to 0$ as $N \to\infty$.
Since
$ \hat{m}_N^* $ and
$\deviation{\hat{\mathscr{S}}_N^{r_N}}{\mathscr{S}}$
are measurable, we obtain the almost sure convergence statements.
\end{proof}

\section{Applications}\label{sec:apps}
We conclude with the application of our main result, 
\Cref{thm:consistency}, to several problem classes.
\subsection{Consistency of Epi-Regularized and Smoothed Empirical
	Approximations}

Using \Cref{thm:consistency}, we
demonstrate the consistency of solutions to epi-regularized and smoothed
risk-averse programs using the average value-at-risk. These types of risk measures are popular in numerical approaches, see \cite{Kouri2016,Kouri2020,Kouri2020a,Beiser2020,Wechsung2021,Curi2020}.
For $\beta \in [0,1)$,
the average value-at-risk
$\AVaR{} : \bochnerreal{1}{\Omega,\cF,P} \to \real$ is
defined by
\begin{align*}
	\AVaR{X} = \inf_{t \in \real}\,
	\{\,t + \tfrac{1}{1-\beta} \cE{\maxo{X-t}} \,\},
\end{align*}
where $\maxo{x} \coloneqq  \max\{0,x\}$ for $x \in \real$.
Throughout the section, $m^*$ and $\mathscr{S}$
denotes the optimal value and $0$-solution set
of \eqref{eq:gen-risk-op},
respectively, with the risk measure $\Risk{} = \AVaR{}$.
Moreover, we denote by $m_N^*$
the optimal value and by $\hat{\mathscr{S}}_N^{r}$
the $r$-solution set ($r \geq 0$) of the problem's empirical approximation.
The average value-at-risk $\AVaR{}$ is a law invariant risk measure
\cite{Shapiro2013a}.

Epi-regularization of risk measures has been proposed and analyzed in
\cite{Kouri2020}. We apply the epi-regularization to
the average value-at-risk.
As in Example~2 in \cite{Kouri2020}, 
we consider $\AVaR{}$ as defined on $\bochnerreal{2}{\Omega,\cF,P}$
throughout the remainder of this section.
We define
$\Phi : \bochnerreal{2}{\Omega,\cF,P} \to \real$ by
\begin{align*}
	\Phi[X] \coloneqq (1/2)\cE{X^2} + \cE{X}.
\end{align*}
For $\varepsilon> 0$, the epi-regularization
$\epiAVaR{}{\varepsilon} \colon \bochnerreal{2}{\Omega,\cF,P} \to \mathbb R$
of  $\AVaR{}$  is given by
\begin{align}
	\label{eq:epiAVaR}
	\epiAVaR{X}{\varepsilon} \coloneqq
	\inf_{Y \in \bochnerreal{2}{\Omega,\cF,P}} \, \{\,
	\AVaR{X-Y} + \varepsilon\Phi[\varepsilon^{-1}Y]
	\,\}.
\end{align}
The risk functional $\epiAVaR{}{\varepsilon}$ can be shown to be law invariant. See Appendix \ref{app:LI}. 

For $\varepsilon > 0$,
we consider the epi-regularized empirical average value-at-risk
optimization problem
\begin{align*}
	\min_{z \in Z_{\rm ad}}\,
	\{\, \epiAVaR{\hat{H}_{\B z,N}}{\varepsilon} + \wp(z)\,\}.
\end{align*}
We let $\hat{m}_{\mathrm{epi},N}^{\varepsilon}$
be its optimal value
and $\hat{\mathscr{S}}^{\varepsilon}_{\mathrm{epi},N}$  be its 
$0$-solution set.
Note that for fixed $\varepsilon > 0$, our main result, \Cref{thm:consistency}, already provides an asymptotic consistency result. However, in numerical procedures, the $\varepsilon$-parameter is typically driven to zero. Therefore, we prove a stronger statement here.
\begin{proposition}
	\label{cor:epiavar}
	Let \Cref{assumption:strongconsistency} hold with $p=2$.
	Suppose further that $(\Omega,\cF, P)$ is nonatomic and complete.
	Let $(\varepsilon_N) \subset (0,\infty)$
	with $\varepsilon_N \to 0$ as $N \to\infty$.
	Then $\hat{m}_{\mathrm{epi},N}^{\varepsilon_N} \to m^*$
	\wpone as $N \to \infty$.
	If furthermore $\wp$ is an R-function, then
	$\deviation{\hat{\mathscr{S}}^{\varepsilon_N}_{\mathrm{epi},N}}{\mathscr{S}} \to 0$ \wpone
	as $N \to\infty$.
\end{proposition}

The proof of \Cref{cor:epiavar} is based on the following result.
\begin{lemma}
	\label{lem:errorepivar}
	Fix $\varepsilon > 0$.
	The functional  $\epiAVaR{}{\varepsilon} : \bochnerreal{2}{\Omega,\cF,P} \to \real$
	is a law invariant, convex risk measure.
	For all
	$X \in \bochnerreal{2}{\Omega,\cF,P}$, it holds that
	\begin{align*}
		\AVaR{X} -\tfrac{\varepsilon\beta}{2(1-\beta)}  \leq
		\epiAVaR{X}{\varepsilon} \leq \AVaR{X}.
	\end{align*}
\end{lemma}

\begin{proof}
The functional $\epiAVaR{}{\varepsilon}$ is  a convex risk measure
\cite[pp.\ 776 and 778--779]{Kouri2020}. By the arguments in Appendix \ref{app:LI}, it is law 
invariant.
Since $\Phi[0] = 0$, the second estimate is implied
by Proposition~1 in \cite{Kouri2020}.
Fix $X \in \bochnerreal{2}{\Omega,\cF,P}$.
Since $\AVaR{}$ is subdifferentiable \cite[p.\ 243]{Shapiro2021},
Proposition~2 in \cite{Kouri2020} yields
for all subgradients $\vartheta \in \partial \AVaR{X}$
(see, e.g., p.\ 480 in \cite{Shapiro2021} 
for definitions of subgradients and subdifferentials),
\begin{align*}
	\epiAVaR{X}{\varepsilon} \geq \AVaR{X} - \varepsilon \Phi^*[\vartheta].
\end{align*}
Here $\Phi^*$ is the Fenchel conjugate to $\Phi$;
see, e.g., p.\ 232 in
\cite{Shapiro2021} for a definition.
Let $\vartheta \in \partial \AVaR{X}$ be arbitrary.
We have $0 \leq \vartheta \leq 1/(1-\beta)$ \wpone,
$\cE{\vartheta} = 1$ \cite[p.\ 243]{Shapiro2021} and
$\Phi^*[\vartheta] = (1/2)\cE{(\vartheta-1)^2}$;
see Remark~5 in \cite{Kouri2020}.
Hence
\begin{align*}
	\Phi^*[\vartheta] = (1/2)\cE{\vartheta^2}
	- \cE{\vartheta} +(1/2)
	= (1/2)\cE{\vartheta^2}
	-(1/2)
	\leq  \frac{1}{2}\frac{1-(1-\beta)}{1-\beta}
	= \frac{1}{2}\frac{\beta}{1-\beta}.
\end{align*}
\end{proof} 

\begin{proof}[{Proof of \Cref{cor:epiavar}}]
Following the proof of \Cref{lem:emp-risk-op-has-solution}
and using the fact that $\epiAVaR{}{\varepsilon_N}$
is a law invariant, convex risk measure
(see \Cref{lem:errorepivar}),
we find that $\hat{m}_{\mathrm{epi},N}^{\varepsilon_N}$ and
$\mathscr{S}_{\mathrm{epi},N}^{\varepsilon_N}$ are measurable.
\Cref{lem:errorepivar} ensures that
$\hat{m}_N^*  -\tfrac{\varepsilon_N\beta}{2(1-\beta)}
\leq  \hat{m}_{\mathrm{epi},N}^{\varepsilon_N} \leq \hat{m}_N^*$.
Applying \Cref{thm:consistency}
with $\Risk{} = \AVaR{}$ yields
$\hat{m}_N^* \to m^*$ \wpone as $N \to \infty$.
Combined with $\varepsilon_N \to 0$,
we find that $\hat{m}_{\mathrm{epi},N}^{\varepsilon_N} \to m^*$
\wpone as $N \to \infty$.

If $z_N^{\varepsilon_N} \in \hat{\mathscr{S}}^{\varepsilon_N}_{\mathrm{epi},N}$, then
\Cref{lem:errorepivar} ensures that
$z_N^{\varepsilon_N} \in \hat{\mathscr{S}}_N^{r_N}$,
where $r_N \coloneqq\tfrac{\varepsilon_N\beta}{2(1-\beta)}$.
Hence $\hat{\mathscr{S}}^{\varepsilon_N}_{\mathrm{epi},N} \subset \hat{\mathscr{S}}_N^{r_N}$, yielding
$\deviation{\hat{\mathscr{S}}^{\varepsilon_N}_{\mathrm{epi},N}}{\mathscr{S}}
\leq \deviation{ \hat{\mathscr{S}}_N^{r_N}}{\mathscr{S}}
$.
Applying \Cref{thm:consistency}
with $\Risk{} = \AVaR{}$ yields the second assertion.
\end{proof} 

Next, we establish the consistency of solutions to smoothed
average value-at-risk problems using a smoothing function
for $\maxo{\cdot}$. For brevity, we focus on a particular smoothing function for
the plus function $\maxo{\cdot}$. For $\varepsilon > 0$,
we define the smoothed
plus function $\maxo{\cdot}_\varepsilon : \real \to \real$ by
\begin{align*}
	\maxo{x}_\varepsilon \coloneqq
	\varepsilon \ln(1+\exp(x/\varepsilon)).
\end{align*}
Using $\maxo{\cdot}_\varepsilon$, we define the smoothed
average value-at-risk
$\smoothAVaR{}{\varepsilon} \colon \bochnerreal{2}{\Omega,\cF,P} \to \real$ by
\begin{align}
	\label{eq:smoothAVaR}
	\smoothAVaR{X}{\varepsilon}
	\coloneqq
	\inf_{t \in \real} \, \{\,
	t + \tfrac{1}{1-\beta} \cE{\maxo{X-t}_\varepsilon}
	\, \}.
\end{align}
This version of the smoothed average value-at-risk has been used in
\cite{Wechsung2021} for stochastic stellarator coil design
and in \cite{Beiser2020} for adaptive sampling techniques
for risk-averse optimization. See the Appendix \ref{app:LI} for a short proof
of its law invariance.

For $\varepsilon > 0$,
we consider the smoothed empirical average value-at-risk
optimization problem
\begin{align*}
	\min_{z \in Z_{\rm ad}}\,
	\{\, \smoothAVaR{\hat{H}_{\B z,N}}{\varepsilon} + \wp(z) \,\},
\end{align*}
We let $\hat{m}_{\mathrm{s},N}^{\varepsilon}$
be its optimal value
and $\hat{\mathscr{S}}^{\varepsilon}_{\mathrm{s},N}$  be its 
$0$-solution set.

\begin{proposition}
	\label{cor:smoothedavar}
	Let \Cref{assumption:strongconsistency} hold with $p=2$.
	Suppose further that $(\Omega,\cF, P)$ is nonatomic and complete.
	Let $(\varepsilon_N) \subset (0,\infty)$
	with $\varepsilon_N \to 0$ as $N \to\infty$.
	Then $\hat{m}_{\mathrm{s},N}^{\varepsilon_N} \to m^*$
	\wpone as $N \to \infty$.
	If furthermore $\wp$ is an R-function, then
	$\deviation{\hat{\mathscr{S}}^{\varepsilon_N}_{\mathrm{s},N}}{\mathscr{S}} \to 0$ \wpone
	as $N \to\infty$.
\end{proposition}
\Cref{cor:smoothedavar} is established using
\Cref{lem:errorsmoothvar}.
\begin{lemma}
	\label{lem:errorsmoothvar}
	Fix $\varepsilon > 0$.
	The functional $\smoothAVaR{}{\varepsilon} \colon \bochnerreal{2}{\Omega,\cF,P} \to \real$
	is a law invariant, convex risk measure. For all  $X \in \bochnerreal{2}{\Omega,\cF,P}$,
	it holds that
	\begin{align*}
		\AVaR{X}   \leq
		\smoothAVaR{X}{\varepsilon} \leq \AVaR{X}
		+ \ln(2)\varepsilon/(1-\beta).
	\end{align*}
\end{lemma}
\begin{proof}
The smoothed average value-at-risk
$\smoothAVaR{}{\varepsilon}$ is a convex risk measure
\cite[Props.\ 4.4--4.6]{Kouri2016}.
By the arguments in Appendix \ref{app:LI}, it is law 
invariant.
For $x \in \real$, we have
$\maxo{x} \leq \maxo{x}_\varepsilon \leq \maxo{x} + \varepsilon\ln(2)$,
yielding the error bounds.
\end{proof} 

\begin{proof}[{Proof of \Cref{cor:smoothedavar}}]
The proof is similar to that of
\Cref{cor:epiavar}.
Following the proof of \Cref{lem:emp-risk-op-has-solution}
and using the fact that $\smoothAVaR{}{\varepsilon_N}$
is a law invariant, convex risk measure
(see \Cref{lem:errorsmoothvar}),
we find that $\hat{m}_{\mathrm{s},N}^{\varepsilon_N}$ and
$\hat{\mathscr{S}}^{\varepsilon_N}_{\mathrm{s},N}$ are measurable.
\Cref{lem:errorepivar} ensures that
$\hat{m}_N^* \leq  \hat{m}_{\mathrm{s},N}^{\varepsilon_N}
\leq \hat{m}_N^* + \ln(2)\varepsilon_N/(1-\beta)$.
Applying \Cref{thm:consistency}
with $\Risk{} = \AVaR{}$ yields
$\hat{m}_N^* \to m^*$ \wpone as $N \to \infty$.
Combined with $\varepsilon_N \to 0$,
we find that $\hat{m}_{\mathrm{s},N}^{\varepsilon_N} \to m^*$
\wpone as $N \to \infty$.

If $z_N^{\varepsilon_N} \in \hat{\mathscr{S}}^{\varepsilon_N}_{\mathrm{s},N}$, 
then
\Cref{lem:errorepivar} ensures that
$z_N^{\varepsilon_N} \in \hat{\mathscr{S}}_N^{r_N}$,
where $r_N \coloneqq\ln(2)\varepsilon_N/(1-\beta)$.
Hence $\hat{\mathscr{S}}^{\varepsilon_N}_{\mathrm{s},N} \subset \hat{\mathscr{S}}_N^{r_N}$, yielding
$\deviation{\hat{\mathscr{S}}^{\varepsilon_N}_{\mathrm{s},N}}{\mathscr{S}}
\leq \deviation{ \hat{\mathscr{S}}_N^{r_N}}{\mathscr{S}}
$.
Applying \Cref{thm:consistency}
with $\Risk{} = \AVaR{}$ yields the second assertion.
\end{proof}

\subsection{Risk-Averse Semilinear PDE-Constrained Optimization}
\label{subsect:ocpspde}
Our consistency result, \Cref{thm:consistency}, is applicable to
risk-averse semilinear PDE-constrained optimization as we demonstrate
in this section. Following \cite{Kouri2020a}
(see also \cite{Garreis2019a,Geiersbach2020}), we consider
\begin{align}
	\label{eq:ocpspde}
	\min_{z \in \adcsp}\,
	(1/2) \Risk{\norm[L^2(\domain)]{\maxo{1-\iota S(z)}}^2}
	+ (\alpha/2) \norm[L^2(\domain)]{z}^2,
\end{align}
where $\alpha > 0$,
$\iota : H^1(\domain) \to L^2(\domain)$
is the embedding operator of the compact embedding
$H^1(\domain) \embedding L^2(\domain)$,
$\adcsp \coloneqq \{\, z \in L^2(\domain) \colon
\mathfrak{l}(x) \leq z(x) \leq \mathfrak{u}(x) \text{ for a.e. } x \in \domain \,\}$
with $\mathfrak{l}$, $\mathfrak{u} \in L^2(\domain)$
and $\mathfrak{l}(x) \leq \mathfrak{u}(x)$
for a.e.\ $x \in \domain$, 
and for each $(z,\xi) \in L^2(\domain) \times \Xi$,
$S(z)(\xi) \in H^1(\domain)$ is the solution to:
\begin{align}
	\label{eq:pde}
	\text{find}\quad u \in H^1(\domain) \colon \quad
	\A(u,\xi)  = \B_1(\xi) \iota^* z + \rhs(\xi),
\end{align}
where $\iota^*$ is the adjoint
operator to $\iota$,
$\A : H^1(\domain) \times \Xi \to H^1(\domain)^*$,
$\B_1 : \Xi \to \spL{H^1(\domain)^*}{H^1(\domain)^*}$,
and $\rhs : \Xi \to H^1(\domain)^*$ are defined by
\begin{align*}
	\dualp[H^1(\domain)]{\A(u,\xi)}{v}
	& \coloneqq \int_\domain
	a(\xi)(x) [\nabla u(x)^T \nabla v(x) + u(x) v(x) ]\du x
	+ \int_\domain u(x)^3v(x) \du x, \\
	\dualp[H^1(\domain)]{\B_1(\xi)y}{v}
	& \coloneqq \int_\domain [B(\xi)y](x)v(x) \du x, \quad
	\dualp[H^1(\domain)]{\rhs(\xi)}{v}
	\coloneqq \int_\domain b(\xi)(x) v(x) \du x.
\end{align*}
Here, $b : \Xi \to L^2(\domain)$ is essentially bounded,
$a : \Xi \to C^0(\bar{\domain})$ is  measurable
and there exist constants $\kappa_{\min}$, $\kappa_{\max}  > 0$
such that
$\kappa_{\min} \leq a(\xi)(x) \leq \kappa_{\max}$ for all
$(\xi,x) \in \Xi \times \bar{\domain}$.
It remains to define $B(\xi) \colon H^1(\domain)^* \to H^1(\domain)$.
Fix $(y,\xi) \in H^1(\domain)^* \times \Xi$.
We define  $B(\xi)y \in H^1(\domain)$ as the solution to:
find $w \in H^1(\domain)$ such that
\begin{align*}
	\int_\domain [r(\xi) \nabla w(x)^T \nabla v(x) + w(x)v(x)]\du x
	= \dualp[H^1(\domain)]{y}{v}
	\quad \text{for all} \quad v \in H^1(\domain),
\end{align*}
where $r : \Xi \to (0,\infty)$ is  random
variable such that there exist $r_{\min}$, $r_{\max}  > 0$
with $r_{\min} \leq r(\xi) \leq r_{\max}$ for all $\xi \in \Xi$.
Since $\iota u = u$
for all $u \in H^1(\domain)$, we have
$\dualp[H^1(\domain)]{\iota^* z}{v} = \inner[L^2(\domain)]{z}{v}$
for all $z \in L^2(\domain)$ and $v \in H^1(\domain)$
\cite[p.\ 21]{Bonnans2000}.

We express \eqref{eq:ocpspde} in the form given in
\eqref{eq:gen-risk-op} and verify \Cref{assumption:strongconsistency}.
For each $(y, \xi) \in H^1(\domain)^* \times \Xi$,
we consider the 	auxiliary random operator equation:
\begin{align}
	\label{eq:auxpde}
	\text{find}\quad u \in  H^1(\domain) \colon \quad
	\A(u,\xi)  = \B_1(\xi)y + \rhs(\xi).
\end{align}
\begin{lemma}
	\label{lem:pde}
	Under the above hypotheses,
	for each $(y, \xi) \in  H^1(\domain)^* \times \Xi$, the
	operator equation
	\eqref{eq:auxpde} has a unique solution $\widetilde{S}(y)(\xi)$,
	$\widetilde{S}(y) \in \bochner{q}{H^1(\domain)}$
	for each $q \in [1,\infty]$ and  $y \in H^1(\domain)^*$,
	$(y,\xi) \mapsto
	\widetilde{S}(z)(\xi)$ is a \Caratheodory\ mapping,
	and $\widetilde{S}: H^1(\domain)^* \to  \bochner{q}{H^1(\domain)}$
	is Lipschitz continuous
	for each $q \in [1,\infty]$.
\end{lemma}
Let $\ssp$ be a reflexive Banach space.
We recall that an operator $A : \ssp \to \ssp^*$
is $\kappa$-strongly monotone if there exists $\kappa > 0$
such that
\[
\dualp[\ssp]{A(u_2)-A(u_1)}{u_2-u_1}
\geq \kappa \norm[\ssp]{u_2-u_1}^2\quad
\text{for all}\quad  u_1,  \, u_2 \in \ssp.
\]
\begin{proof}[{Proof of \Cref{lem:pde}}]
For each $\xi \in \Xi$, $\A(\cdot,\xi)$
is $\kappa_{\min}$-strongly monotone
and it holds that
\[\norm[\spL{H^1(\domain)^*}{H^1(\domain)^*}]
{\B_1(\xi)} \leq 1/\min\{r_{\min},1\};
\]
cf.\ \cite[p.\ 13]{Kouri2020a}.
The existence, uniqueness and the stability estimate
\[
\norm[H^1(\domain)]{\widetilde{S}(y)(\xi)}
\leq (1/\kappa_{\min}) \norm[H^1(\domain)^*]{\B_1(\xi) y} +
(1/\kappa_{\min}) \norm[H^1(\domain)^*]{\rhs(\xi)}
\]
are a consequence of
the Minty--Browder theorem \cite[Thm.\ A.26]{Zeidler1990},
for example. Using Filippov's theorem \cite[Thm.\ 8.2.10]{Aubin2009},
we can show that $\widetilde{S}(y)$ is measurable.
Combined with the stability estimate and H\"older's inequality,
we conclude that
$\widetilde{S}(y) \in \bochner{q}{H^1(\domain)}$
for each $q \in [1,\infty]$ and $y \in H^1(\domain)^*$.
Since  for all 	$y_1$, $y_2 \in H^1(\domain)^*$ and $\xi \in \Xi$,
we have (cf.\ \cite[eq.\ (3.7)]{Kouri2020a})
\[
\norm[\ssp]{\widetilde{S}(y_2)(\xi)-\widetilde{S}(y_1)(\xi)}
\leq (1/\kappa_{\min}) \norm[H^1(\domain)^*]{\B_1(\xi)[y_2-y_1]}
,
\]
the mapping 
$(y,\xi) \mapsto \widetilde{S}(y)(\xi)$
is a \Caratheodory\ mapping, and
$\widetilde{S} : H^1(\domain)^* \to \bochner{q}{H^1(\domain)}$
is Lipschitz continuous for all $q \in [1,\infty]$.
\end{proof}

The function $\wp$ defined by
$\wp(z) \coloneqq (\alpha/2)\norm[L^2(\domain)]{z}^2$
is an R-function according to \Cref{lem:rfunction}, as $\alpha > 0$
and $L^2(\domain)$ is a Hilbert space and hence
has the Radon--Riesz property \cite[Prop.\ 2.35]{Bonnans2000}.
The operator $\B \coloneqq \iota^*$ is linear and completely continuous
because $\iota$ is a compact operator by the Sobolev embedding theorem.
We define  $\pobj : H^1(\domain)^* \times \Xi \to [0,\infty)$ by
$\pobj(y,\xi) \coloneqq
(1/2)\norm[L^2(\domain)]{\maxo{1-\iota \widetilde{S}(y)(\xi)}}^2$.
The mapping $\mathcal{J} : L^{q}(\Xi, \cA, \mathbb{P}; H^1(\domain)) \to 
L^{q/2}(\Xi, \cA, \mathbb{P})$ 
given by $\mathcal{J}(u) \coloneqq 
(1/2)\norm[L^2(\domain)]{\maxo{1-\iota u}}^2$ is continuous
for $q \in [2,\infty)$ \cite[Proposition 5]{Kouri2020}.
\Cref{lem:pde} ensures that
$\mathcal{J} \circ \widetilde{S} \colon H^1(\domain)^* \to 
L^{q/2}(\Xi, \cA, \mathbb{P})$
is well-defined and  continuous for $q \in [2,\infty)$, 
yielding the continuity of $F$ with $p = q/2$.
Having verified \Cref{assumption:strongconsistency}
for $p \in [2,\infty)$, we can apply
\Cref{thm:consistency} to study the consistency of
empirical approximations of \eqref{eq:ocpspde}.

\subsection{Risk-Averse Optimization with  Variational Inequalities}
\label{subsect:ocpsvi}

We consider a risk-averse optimization problem governed
by an  elliptic  variational inequality with random inputs.
Our presentation  is inspired by that in \cite{Hertlein2022}.
We consider
\begin{align}
	\label{eq:ocpvarineq}
	\min_{z \in \adcsp} \,
	(1/2) \Risk{\norm[L^2(\domain)]{\iota S(z)-u_d}^2}
	+ (\alpha/2) \norm[L^2(\domain)]{z}^2,
\end{align}
where $\alpha > 0$, $u_d \in L^2(\domain)$,
$\iota : H_0^1(\domain) \to L^2(\domain)$
is the embedding operator of the compact embedding
$H_0^1(\domain) \embedding L^2(\domain)$,
and $\adcsp$ is as in \Cref{subsect:ocpspde}. For each
$(z,\xi) \in L^2(\domain) \times \Xi$, $S(z)(\xi) \in H_0^1(\domain)$ is the
solution to the parameterized elliptic variational inequality:
\begin{align}
	\label{eq:varineq}
	\text{find} \quad u \in K_\psi \colon \quad
	\dualpHzeroone[\domain]{A(\xi)u-\iota^* z}{v-u} \geq 0
	\quad \text{for all} \quad
	v \in K_\psi,
\end{align}
where $\iota^*$ is the adjoint operator to $\iota$,
$H^{-1}(\domain) \coloneqq H_0^1(\domain)^*$,
$A : \Xi \to \spL{H_0^1(\domain)}{H^{-1}(\domain)}$
is a parameterized elliptic operator,
and
$K_\psi \coloneqq \{\, u \in H_0^1(\domain) \colon \, 
u(x) \geq \psi(x) \text{ for a.e. } x \in \domain\,\}$
with $\psi \in H^1(\domain)$
and $\psi_{\partial \domain} \leq 0$ is the  obstacle.
The set $K_\psi$ is nonempty \cite[p.\ 129]{Surowiec2018}.
For $(y,\xi) \in H^{-1}(\domain) \times \Xi$,
we also consider the auxiliary parameterized elliptic variational inequality:
\begin{align}
	\label{eq:auxvarineq}
	\text{find} \quad u \in K_\psi \colon \quad
	\dualpHzeroone[\domain]{A(\xi)u-y}{v-u} \geq 0
	\quad \text{for all} \quad
	v \in K_\psi.
\end{align}
If $\widetilde{S}(y)(\xi)$ with $y = \iota^*z$ is a solution to
\eqref{eq:auxvarineq}, then it is a solution to \eqref{eq:varineq}.

We assume that $A : \Xi \to \spL{H_0^1(\domain)}{H^{-1}(\domain)}$
is uniformly measurable, that is,
there exists a sequence $A_k : \Xi \to \spL{H_0^1(\domain)}{H^{-1}(\domain)}$
of simple mappings such that $A_k(\xi) \to A(\xi)$
in $\spL{H_0^1(\domain)}{H^{-1}(\domain)}$
as $k \to \infty$ for each $\xi \in \Xi$. Moreover, we assume that
there exist
constants $\kappa_{\min}$, $\kappa_{\max}  > 0$
such that for each $\xi \in \Xi$, $A(\xi)$ is $\kappa_{\min}$-strongly monotone
and $\norm[\spL{H_0^1(\domain)}{H^{-1}(\domain)}]{A(\xi)} \leq \kappa_{\max}$.
Under these conditions,
the auxiliary variational inequality \eqref{eq:varineq} has a unique
solution $\widetilde{S}(y)(\xi)$
for each $(y,\xi) \in H^{-1}(\domain) \times \Xi$,
and  $\widetilde{S}(\cdot)(\xi)$ is Lipschitz continuous
with Lipschitz constant $1/\kappa_{\min}$
for each $\xi \in \Xi$; cf.\ \cite[Thm.\ 7.3]{Hertlein2022}.
Using results established in \cite[p.\ 180]{Gwinner2022}, we can show that
$\widetilde{S}(y) \in \bochner{q}{H_0^1(\domain)}$
for all $q \in [1,\infty]$ and $y \in H^{-1}(\domain)$.
Combined with the Lipschitz continuity, we find that
$\widetilde{S} : H^{-1}(\domain) \to  \bochner{q}{H_0^1(\domain)}$
is continuous for each $q \in [1,\infty]$.

We express \eqref{eq:ocpvarineq} in the form given in
\eqref{eq:gen-risk-op} and verify \Cref{assumption:strongconsistency}.
The function $\wp$ defined by
$\wp(z) \coloneqq (\alpha/2)\norm[L^2(\domain)]{z}^2$
is an R-function; see \Cref{subsect:ocpspde}.
The operator $\B \coloneqq \iota^*$ is linear and completely continuous
because $\iota$ is a compact operator.
We define  $\pobj : H^{-1}(\domain) \times \Xi \to [0,\infty)$ by
$\pobj(y,\xi) \coloneqq (1/2)\norm[L^2(\domain)]{\iota \widetilde{S}(y)(\xi) -u_d}^2$.
The mapping 
$\mathcal{J} : L^{q}(\Xi, \cA, \mathbb{P}; H_0^1(\domain)) \to 
L^{q/2}(\Xi, \cA, \mathbb{P})$ given by 
$\mathcal{J}(u) \coloneqq (1/2)\norm[L^2(\domain)]{\iota u -u_d}^2$ is continuous for $q \in [2,\infty)$; cf.\
\cite[Example 3.2 and Theorem 3.5]{Kouri2018a}. 
Combined with the continuity of $\widetilde{S}$,
we find that
$\mathcal{J} \circ \widetilde{S}
\colon H^{-1}(\domain) \to L^{q/2}(\Xi, \cA, \mathbb{P})$
is well-defined and  continuous for $q \in [2,\infty)$, yielding the continuity
of $F$ with $p = q/2$.
Having verified \Cref{assumption:strongconsistency}
for $p \in [1,\infty)$, we can apply
\Cref{thm:consistency}, which in turn yields the consistency of
empirical approximations of \eqref{eq:ocpvarineq}.

\section{Conclusion}\label{sec:conclusion}
We have seen that consistency results, in particular,
norm consistency of empirical minimizers for nonconvex, risk-averse stochastic optimization problems involving infinite dimensional decision spaces are in fact available.  The central property on which the entire discussion depends is the ability to draw compactness from the structure of the objective function. As the examples illustrate, this is much more the rule rather than the exception. In fact, even in examples such as topology optimization, \cite{MPBendsoe_OSigmund_2003}, where the decision variable enters the PDE in a nonlinear fashion, the required use of either filters or other regularization strategies, see e.g.\ \cite{BSLazarov_OSigmund_2011,Sigmund2013}, also provides compactness. 

There remain many open challenges. These include applications to multistage or dynamic problems, large deviation results for optimal values
and solutions, and central limit theorems. In many instances, the known techniques are limited by nonsmoothness of the risk measure $\risk$ and the infinite dimensional decision spaces. However, the main result in this text,
\Cref{thm:consistency}, is a first major step and an essential tool towards
verifying the convergence of numerical optimization methods that make use of empirical approximations.
Moreover, for numerical computations, the decision spaces
of infinite dimensional  risk-averse optimization problems must
typically be discretized. Therefore, in a practical setting, these problems
have the additional challenge that the numerically computed estimators are generally dependent on both the sample size $N$ and additional 
spacial discretization parameters. As initial contributions 
for risk-neutral PDE-constrained problems   
\cite{MHoffhues_WRoemisch_TMSurowiec_2021,Milz2022c}
demonstrate, 
an infinite dimensional consistency analysis provides an important component in the numerical analysis of these challenging optimization problems.

\appendix 
	\section{Law of Large Numbers for Risk Functionals}
	
	We generalize the epigraphical law of large numbers for law invariant
	risk function established in Theorem~3.1 in
	\cite{Shapiro2013} to allow for
	random lower semicontinuous functions defined on complete, separable
	metric spaces instead of $\real^n$. 
	The proof of Theorem~3.1 provided in \cite{Shapiro2013}
	generalizes to this more general setting with
	only a few notational changes needed.
	Nevertheless, we verify the liminf-condition
	of epiconvergence using ideas from the proof
	of Proposition~7.1  in \cite{Royset2020}.
	The limsup-condition is established as in  \cite{Shapiro2013}.

	\begin{assumption}
		\label{assumption:llnrm}
		\normalfont
		Let
		$(\Omega, \cF, P)$ be a  nonatomic, complete
		probability space,
		and let $(\Theta, \Sigma, \mathbb{M})$ be
		a complete probability space.
		Let $\zeta : \Omega \to \Theta$ be a random element 
		with distribution $\mathbb{M}$ and let
		$\zeta^1, \zeta^2, \ldots$ defined on 
		a complete probability space
		$(\Omega^\prime, \cF^\prime, P^\prime)$
		be independent identically distributed $\Theta$-valued random elements
		each having the same distribution as that of $\zeta$.
		Let $(V, d_V)$ be a complete, separable metric space and let
		$1 \leq p < \infty$.
		\begin{enumerate}[label=($\mathfrak{R}${\arabic*})]
			\item
			\label{itm:llnR2}
			The function $\Psi: V \times \Theta
			\to \real$ is random lower semicontinuous.
			\item
			\label{itm:llnR1}
			For each
			$v \in V$, $\Psi_v(\cdot) \coloneqq
			\Psi(v, \cdot) \in L^p(\Theta, \Sigma, \mathbb{M})$.
			\item
			\label{itm:llnR3}
			For each $\bar v \in V$, there exists a neighborhood
			$\mathcal{V}_{\bar v} \subset V$ of $\bar v$
			and a random variable $h \in L^p(\Theta, \Sigma, \mathbb{M})$
			such that $\Psi(v, \cdot) \geq h(\cdot)$
			for all $v \in \mathcal{V}_{\bar v}$.
		\end{enumerate}
	\end{assumption}
	
	\Cref{thm:lln31} is as Theorem~3.1  in \cite{Shapiro2013} but allows for
	complete, separable metric spaces $V$ instead of $\real^n$.
	Let $\rho : \Lp{p} \to \real$ be a law invariant
	risk measure and let \Cref{assumption:llnrm} hold true.
	Let $v \in V$ and let 
	$\hat{H}_{v,N}(\cdot; \omega^\prime)$
	be the empirical distribution function
	of $\Psi(v, \zeta^1(\omega^\prime)), \ldots, 
	\Psi(v,\zeta^N(\omega^\prime))$.
	Moreover, let $\hat{H}_{v,N}^{-1}(\cdot; \omega^\prime)$
	be its quantile function. 
	We define $\hat{\phi}_N : V \times \Omega^\prime \to \real$
	and $\phi : V \to \real$ by
	$\hat{\phi}_N(v,\omega^\prime) \coloneqq 
	\rho(\hat{H}_{v,N}^{-1}(G(\cdot);\omega^\prime))$
	and $\phi(v) \coloneqq \rho(\Psi_v(\zeta)) = 
	\rho(\Psi_v(\zeta(\cdot)))
	$.
	Here $G \colon \Omega \to [0,1]$ is a random variable
	with uniform distribution $\nu$ as discussed in \Cref{ssec:spaces}.
	We often omit writing the second argument of $\hat{\phi}_N$.
	
	\begin{theorem}
		\label{thm:lln31}
		If \Cref{assumption:llnrm} holds and
		$\rho : \Lp{p} \to \real$
		is a law invariant, convex risk measure,
		then $\phi$ is lower semicontinuous
		and finite-valued, and $\hat{\phi}_N$
		epiconverges to $\phi$ \wpone as $N \to\infty$.
	\end{theorem}

	Before establishing \Cref{thm:lln31}, we formulate a
	law of large numbers with respect to Mosco-epiconvergence.
	\begin{corollary}
		\label{cor:lln31}
		Let $Y_0 \subset Y$ be a closed subset
		of a separable Banach space $Y$ and let 
		$W_0$ be a nonempty, closed, convex
		subset of a reflexive, separable Banach space $W$.
		Let the hypotheses of \Cref{thm:lln31} hold
		with $V = Y_0$. Suppose that
		$\B: W \to Y$ is linear and completely continuous
		with $\B(W_0) \subset Y_0$.
		Then $\hat{\phi}_N \circ \B : W_0 \to \real$ Mosco-epiconverges to
		$\phi \circ \B : W_0 \to \real$ \wpone as $N \to\infty$.
	\end{corollary}
	\begin{proof}
	\Cref{thm:lln31} ensures that
	$\hat{\phi}_N$
	epiconverges to $\phi$ \wpone as $N \to\infty$.
	Since $W_0$ defines a complete separable metric space, 
	$\B(W_0) \subset Y_0$,
	and $\B$ is continuous, \Cref{thm:lln31} further ensures that
	$\hat{\phi}_N \circ \B$
	epiconverges to $\phi \circ \B$ \wpone as $N \to\infty$.
	Combined with \Cref{prop:moscoepiconvergence}
	and the complete continuity of $\B$, we conclude
	that $\hat{\phi}_N \circ \B$ Mosco-epiconverges to
	$\phi \circ \B$ \wpone as $N \to\infty$.
	\end{proof}

	As already mentioned, the proof of \Cref{thm:lln31}
	presented in \cite[Thm.\ 3.1]{Shapiro2013}
	for $V = \real^n$ can be generalized to the above setting
	without much effort.
	A key result for establishing \Cref{thm:lln31}
	is \Cref{thm:lln21}.
	To formulate \Cref{thm:lln21}, let $X \in L^p(\Omega,\cF,P)$
	be a random variable and $X_1, X_2, \ldots$ defined on
	a complete probability space
	$(\Omega^\prime, \cF^\prime, P^\prime)$ be independent identically distributed
	real-valued random variables each having the same distribution
	as that of $X$. Moreover, let
	$\hat{H}_{N}(\cdot; \omega^\prime)$ be the empirical distribution function
	of the sample $X_1(\omega^\prime), \ldots, X_N(\omega^\prime)$,
	and  $\hat{H}_N^{-1}(\cdot; \omega^\prime)$ be its quantile function. 
	
	\begin{theorem}[{see \cite[Thm.\ 2.1]{Shapiro2013} and
			\cite[Thm.\ 9.65]{Shapiro2021}}]
		\label{thm:lln21}
		If $(\Omega, \cF, P)$ is complete and nonatomic,
		$1 \leq p < \infty$, and
		$\rho : \Lp{p} \to \real$
		is a law invariant, convex risk measure,
		then 
		$\rho(\hat{H}_N)$ converges to 
		$\rho(X)$ \wpone as
		$N \to\infty$.
	\end{theorem}
	
	\begin{proof}
	We present a proof somewhat different from that in 
	\cite{Shapiro2013}. Fix $\omega^\prime \in \Omega^\prime$.
	Using a change of variables and the fact that
	$G \colon \Omega \to [0,1]$ has uniform distribution $\nu$, we obtain
	\begin{align*}
		\int_{\Omega} |\hat{H}_{N}^{-1}(G(\omega); \omega^\prime)-
		H_X^{-1}(G(\omega))|^p \du P(\omega)
		&= 
		\int_{0}^1 |\hat{H}_{N}^{-1}(q; \omega^\prime)-
		H_X^{-1}(q)|^p \du P\circ G^{-1}(q)
		\\
		&= 
		\int_{0}^1 |\hat{H}_{N}^{-1}(q; \omega^\prime)-
		H_X^{-1}(q)|^p \du \nu (q).
	\end{align*}
	Since $1 \leq p < \infty$, $X \in L^p(\Omega,\cF, P)$
	and $X_1, X_2, \ldots$ defined on $(\Omega^\prime, \cF^\prime, P^\prime)$ are independent 
	identically distributed each with the same
	distribution as that of $X$,
	the latter integral converges
	$P^\prime$-almost surely to zero as $N \to \infty$;
	see \cite[Cor.\ on p.\ 48]{Rachev1985},
	\cite[Cor.\ 3]{Rachev1982a},
	\cite[Cor.\ 3 on p.\ 666]{Rachev1985a}.
	Since $\rho \colon \Lp{p} \to \real$ is a real-valued
	convex risk measure, it is continuous
	\cite[Cor.\ 3.1]{Ruszczynski2006b}. We obtain
	for almost every $\omega^\prime \in \Omega^\prime$,
	$
	\rho(\hat{H}_{N}^{-1}(G(\cdot); \omega^\prime))
	\to \rho(H_X^{-1}(G(\cdot)))
	$ as $N \to \infty$.
	Combined with $\rho(X) = \rho(H_X^{-1}(G(\cdot)))$ and
	$\rho(\hat{H}_N) = \rho(\hat{H}_N^{-1}(G(\cdot)))$,
	we obtain the assertion. 
	\end{proof}

	\begin{proof}[{Proof of \Cref{thm:lln31}}]
	The fact that $\phi$ is finite-valued
	and lower semicontinuous can be established
	as in the proof of Theorem~3.1 in \cite{Shapiro2013}.
	To establish the epiconvergence, we make use of the constructions
	made in the proof of Proposition~7.1 in \cite{Royset2020}.
	Proposition~7.1 in \cite{Royset2020} establishes epiconvergence
	in case that $\rho(\cdot)  = \cE{\cdot}$, but without assuming
	$(\Omega,\cF,P)$ be nonatomic.
	Let $\mathcal{E} \subset V$
	be a countable dense subset of $V$ and
	$Q_+$ be the nonnegative rational numbers. For $v \in 
	V$
	and $r \in [0,\infty)$, we define 
	$\pi_{v,r}$ on $\Theta$ by
	\begin{align*}
		\pi_{v,r}(\theta) \coloneqq \inf_{w\in B(v,r)}\, 
		\Psi(w, \theta)
		\quad \text{if} \quad r > 0
		\quad \text{and} \quad
		\pi_{v,0}(\theta) \coloneqq \Psi(v,\theta)
		\quad \text{if} \quad r = 0,
	\end{align*}
	where $B(v,r) \coloneqq \{w \in V \colon d_V(w,v) < r\}$.
	Theorem~3.4 in \cite{Korf2001}, \ref{itm:llnR2}, and
	\ref{itm:llnR1} ensure that $\pi_{v,r}$ is an
	extended real-valued 
	random variable for each $v \in V$ and $ r \geq 0$.
	Combined with \ref{itm:llnR3}, we find that for every 
	$v \in \mathcal{E}$,
	there exists a neighborhood $\mathcal{V}_v \subset V$ of $v$ and
	$r_v \in (0,\infty)$ such that
	\begin{align*}
		B(v,r_v) \subset \mathcal{V}_v
		\quad \text{and} \quad
		\pi_{v,r}(\cdot) \in \bochnerreal{p}{\Theta, \Sigma, \mathbb{M}}
		\quad \text{for all} \quad r\in [0,r_v] \cap Q_+.
	\end{align*}

	Let $\tilde H_{v,r,N}(\cdot; \omega^\prime)$
	be the empirical distribution function of 
	$\pi_{v,r}(\zeta^i(\omega^\prime))$, $i = 1, \ldots, N$,
	and let $\tilde H_{v,r,N}^{-1}(\cdot; \omega^\prime)$
	be its quantile function.
	For every $v \in \mathcal{E}$ and
	$r \in [0,r_v] \cap Q_+$,
	\Cref{thm:lln21} ensures that
	$\rho(\tilde H_{v,r,N}^{-1}(G;\cdot)) \to 
	\rho(\pi_{v,r}(\zeta))$
	\wpone as $N \to\infty$.
	Since
	$\{(v,r) \colon \, r\in
	[0,r_v] \cap Q_+ , \, v \in \mathcal{E}\}$ is countable,
	there exists $\Omega_0^\prime \subset \Omega^\prime$ with $
	\Omega_0^\prime \in \cF^\prime$
	and $P^\prime(\Omega_0^\prime) = 1$ such that
	\begin{align*}
		\rho(\tilde H_{v,r,N}^{-1}(G;\omega^\prime)) 
		\to \rho(\pi_{v,r}(\zeta))
		\quad \text{as} \quad N \to\infty
		\quad \text{for all} \quad
		\omega^\prime \in \Omega_0^\prime
		\quad \text{and} \quad
		r \in [0,r_v] \cap Q_+,\, v \in \mathcal{E}.
	\end{align*}
	Now, we verify the liminf-condition of epiconvergence.
	Fix $v \in V$ and fix $v_N \to v$ as $N \to\infty$.
	There exist
	$z_{\ell} \in \mathcal{E}$
	with $z_{\ell} \to v$ as $\ell \to \infty$,
	$r_\ell \in (0,r_v] \cap Q_+$ with $r_\ell \to 0$,
	and for each $\ell \in \natural$, there exists
	$\bar{N}(\ell) \in \mathbb{N}$ such that
	\begin{align*}
		v \in B(z_{\ell+1}, r_{\ell+1}) \subset B(z_{\ell}, r_\ell),
		\quad \text{and} \quad
		v_N \in B(z_{\ell}, r_\ell) \quad
		\text{for all} \quad N \geq \bar{N}(\ell).
	\end{align*}

	Fix $\ell \in \mathbb{N}$.
	For all $N \geq \bar{N}(\ell)$ and $\omega^\prime \in \Omega_0^\prime$,
	Theorem~6.50 in \cite{Shapiro2021} when combined
	with the fact that $\rho$ is law
	invariant and monotone ensures
	\begin{align}
		\label{eq:llnlowerbound}
		\hat{\phi}_N(v_N,\omega^\prime) = 
		\rho(\hat H^{-1}_{v_N,N}(G; \omega^\prime)) \geq 
		\rho(\tilde H_{z_{\ell},r_\ell, N}^{-1}(G;\omega^\prime)).
	\end{align}
	Moreover, for all $\omega^\prime \in \Omega_0^\prime$,
	\begin{align}
		\label{eq:llnlowerbound2}
		\rho(\tilde H_{z_{\ell},r_\ell, N}^{-1}(G;\omega^\prime))
		\to \rho(\pi_{z_{\ell},r_\ell}(\zeta))
		\quad \text{as} \quad N \to\infty.
	\end{align}
	Since $v \in 
	B(z_{\ell+1}, r_{\ell+1}) \subset B(z_{\ell}, r_\ell)$,
	we have $\pi_{z_{\ell},r_\ell} \leq 
	\pi_{z_{\ell+1}, r_{\ell+1}} \leq  \pi_{v,0}$.
	For all $\ell \in \mathbb{N}$ and $\theta \in \Theta$,
	the lower semicontinuity
	of $\Psi(\cdot,\theta)$ (see \ref{itm:llnR2})
	ensures
	$\pi_{z_{\ell},r_\ell}(\theta)  \nearrow
	\pi_{v,0}(\theta) = \Psi(v,\theta)$
	as $\ell \to\infty$
	\cite[p.\ 432]{Korf2001}.
	Thus
	$\pi_{v,0}-\pi_{z_{1}, r_{1}} \geq
	\pi_{v,0}-\pi_{z_{\ell+1}, r_{\ell+1}} \geq 0$
	for all $\ell \in \mathbb{N}$.
	Consequently,
	$|\pi_{v,0}-\pi_{z_{1}, r_{1}}|^p \geq
	|\pi_{v,0}-\pi_{z_{\ell+1}, r_{\ell+1}}|^p$.
	Since
	$\pi_{z_{1},r_1}(\zeta)$, $\pi_{v,0}(\zeta) \in  \Lp{p}$,
	the dominated convergence theorem
	implies
	$\pi_{z_{\ell},r_\ell}(\zeta)\to \pi_{v,0}(\zeta)$
	as $\ell \to\infty$ in $\Lp{p}$.
	Using the fact that the risk measure $\rho$ is real-valued and convex,
	it follows that $\rho$ is continuous \cite[Cor.\ 3.1]{Ruszczynski2006b}
	and monotone.
	Consequently,
	$\rho(\pi_{z_{\ell},r_\ell}(\zeta)) \nearrow 
	\rho(\pi_{v,0}(\zeta)) = \phi(v)$
	as $\ell \to\infty$.
	Combined with \eqref{eq:llnlowerbound} and \eqref{eq:llnlowerbound2},
	we find that 
	for all $\omega^\prime \in \Omega_0^\prime$,
	\begin{align*}
		\liminf_{N\to \infty}\,
		\hat{\phi}_N(v_N,\omega^\prime) \geq \phi(v).
	\end{align*}

	Now, we verify the limsup-condition of epiconvergence
	using the arguments in \cite{Shapiro2013}.
	Since $\phi$ is defined on a separable metric space,
	finite-valued and lower semicontinuous, there exists
	a countable set $\mathcal{D} \subset V$ such that
	for each $v \in V$, there exists a sequence
	$(v_k) \subset \mathcal{D}$ such that
	$v_k \to v$ and $\phi(v_k) \to \phi(v)$ as $k \to\infty$
	\cite[Lem.\ 3]{Zervos1999}.
	Since $\mathcal{D}$ is countable, \Cref{thm:lln21} ensures
	the existence of $\Omega_1^\prime \subset \Omega^\prime$
	with $\Omega_1^\prime \in \cF^\prime$ and 
	$P^\prime(\Omega_1^\prime) = 1$ such that
	for each $v \in \mathcal{D}$ and all
	$\omega^\prime \in \Omega_1^\prime$, we have
	$\hat{\phi}_N(v,\omega^\prime) \to \phi(v)$.
	Fix $v \in V$ and let
	$(v_k) \subset \mathcal{D}$ be a sequence such that
	$v_k \to v$ and $\phi(v_k) \to \phi(v)$ as $k \to\infty$.
	We now proceed with a diagonalization argument
	(see, e.g., Corollary~1.16 or 1.18 in \cite{Attouch1984}).
	For each $k \in \natural$
	and every $\omega^\prime \in \Omega_1^\prime$,
	we have $\hat{\phi}_N(v_k,\omega^\prime) \to \phi(v_k)$
	as $N \to\infty$. Moreover
	$\phi(v_k) \to \phi(v)$ as $k \to\infty$.
	Consequently, for each $\omega^\prime \in \Omega_1^\prime$,
	there exists a mapping
	$\mathbb{N} \ni N \mapsto k_{\omega^\prime}(N) \in \mathbb{N}$
	increasing to $\infty$ such that
	$\hat{\phi}_N(v_{k_{\omega^\prime}(N)},\omega^\prime) \to \phi(v)$
	as $N \to\infty$. Since $v_k \to v$  as $k\to\infty$,
	we further have
	$v_{k_{\omega^\prime}(N)} \to v$ as $N \to\infty$
	for each $\omega^\prime \in \Omega_1^\prime$.
	Combining the derivations, we have shown that
	for each $\omega^\prime \in \Omega_1^\prime$ and every
	$v \in V$, there exists a sequence 
	$(v_{k_{\omega^\prime}(N)})$ converging to $v$
	as $N \to\infty$
	and $\hat{\phi}_N(v_{k_{\omega^\prime}(N)},\omega^\prime) \to \phi(v)$
	as $N \to\infty$.
	Since $\Omega_0^\prime \cap \Omega_1^\prime \in \mathcal{F}^\prime$
	and $P^\prime(\Omega_0^\prime \cap \Omega_1^\prime)=1$, we have demonstrated
	the almost sure epiconvergence of $\hat{\phi}_N$ to
	$\phi$.
	\end{proof}

	\section{Law Invariance of $\epiAVaR{}{\varepsilon}$ and  $\smoothAVaR{}{\varepsilon}$}\label{app:LI}%
	Both  $\epiAVaR{}{\varepsilon}$ defined in \eqref{eq:epiAVaR} and  $\smoothAVaR{}{\varepsilon}$ given in \eqref{eq:smoothAVaR} are optimized certainty equivalents in the sense of \cite{ABenTal_MTeboulle_2007a}, i.e.\ they are fully characterized by convex,
	continuous scalar regret functions $v_{\rm epi,\varepsilon}, v_{\mathrm{s},\varepsilon} : \mathbb R \to \mathbb R$ such that for each $ X \in L^2(\Omega,\mathcal{F},P)$,
	\[
	\aligned
	\epiAVaR{X}{\varepsilon}
	&=
	\inf_{t \in \mathbb R} \{\, t + \mathbb E[v_{\rm epi,\varepsilon}(X-t)]\, \},\\
	\smoothAVaR{X}{\varepsilon}
	&=
	\inf_{t \in \mathbb R} \{\, t +  \mathbb E[v_{\mathrm{s},\varepsilon}(X-t)]\, \},
	\endaligned
	\]
	where $\varepsilon > 0$, 
	$v_{\mathrm{s},\varepsilon}(x) \coloneqq (1-\beta)^{-1}\maxo{x}_\varepsilon$,
	and 
	\begin{align*}
		v_{\rm epi,\varepsilon}(x)
		\coloneqq
		\begin{cases}
			-\frac{\varepsilon}{2} & \text{if} \quad x \in (-\infty,-\varepsilon], \\
			\tfrac{1}{2\varepsilon}
			x^2+x \quad 
			& \text{if} \quad x \in
			\big(-\varepsilon, \tfrac{\varepsilon\beta}{1-\beta}\big),
			\\
			\tfrac{1}{1-\beta}\big(x-\tfrac{\varepsilon\beta^2}{2(1-\beta)}\big)
			& \text{otherwise}.
		\end{cases}
	\end{align*}
	The fact that $\epiAVaR{}{\varepsilon}$ can be expressed in the above
	form has been demonstrated in Example~2
	on p.~778 in \cite{Kouri2020}.

	It is not essential for the underlying probability space to be nonatomic for the law invariance of these functionals. Indeed, start by letting $v : \real \to \real$ be continuous and hence, measurable.
	For each $X \in \Lp{2}$ and $t \in \real$,
	let $v(X-t)$ be integrable, which is the case for both $v_{\rm epi,\varepsilon}$ and 
	$v_{\mathrm{s},\varepsilon}$.
	Let $X_1$, $X_2 \in \Lp{2}$ be distributionally equivalent
	with respect to $P$. Since the distribution
	functions of $X_1$ and $X_2$ are equal
	and each distribution function
	uniquely determines a probability law on $\real$
	\cite[Thm.\ 12.4]{Billingsley2012},
	it holds that $P \circ X_1^{-1} = P \circ X_2^{-1}$.
	For all $t \in \real$, we have
	\begin{align*}
		\cE{v(X_1-t)}
		&= \int_\Omega v(X_1(\omega)-t) \du P(\omega)
		= \int_{\real} v(x-t) \du P \circ X_1^{-1}(x)
		\\
		&= \int_{\real} v(x-t) \du P \circ X_2^{-1}(x)
		= \int_\Omega v(X_2(\omega)-t) \du P(\omega)
		= \cE{v(X_2-t)}.
	\end{align*}
	Hence, $\epiAVaR{}{\varepsilon}$
	and $\smoothAVaR{}{\varepsilon}$ are law invariant. As a result, a large class of risk measures/optimized certainty equivalents are law invariant.

\section*{Acknowledgments.}
JM is very grateful to Prof.\ Alexander Shapiro of several discussions
about empirical approximations.  
We would like to thank Prof.\ Darinka Dentcheva for several discussions regarding convergence of empirical quantile functions and additional tips on the literature. We thank the two anonymous reviewers for their helpful comments and suggestions.

\bibliography{JMilz_TMSurowiec_2022a_v2.bbl}

\end{document}